\documentclass[12pt,reqno]{NumPDEsArticle}

\usepackage[utf8]{inputenc}
\usepackage[T1]{fontenc}
\usepackage[english]{babel}
\usepackage{csquotes}

\usepackage{NumPDEsMacros}
\usepackage{amsmath,amssymb}
\usepackage{enumitem}
\usepackage{graphicx}
\usepackage[exponent-product=\cdot]{siunitx}
\usepackage{biblatex}

\usepackage{tikz}
\usetikzlibrary{calc,angles,quotes,positioning}

\newcommand{\constvec}[1]{\vec{\mathrm #1}}

\renewcommand{\u}{\vec{u}}
\newcommand{\w}{\vec{w}}
\newcommand{\e}{\constvec{e}}
\newcommand{\z}{\constvec{z}}
\renewcommand{\v}{\vec{v}}
\newcommand{\n}{\vec{n}}

\newcommand{\colvec}[3]{\begin{pmatrix}#1\\#2\\#3\end{pmatrix}}

\newcommand{\weak}{\rightharpoonup}
\newcommand{\helical}{{\mathfrak{h}}}
\renewcommand{\leq}{\leqslant}
\renewcommand{\geq}{\geqslant}

\newcommand{\lex}{\ell_\textrm{ex}}
\newcommand{\transposed}{\intercal}

\let\oldtimes\times
\renewcommand{\times}{\!\oldtimes\!}

\makeatletter
\newcommand{\subref}[2]{\textup{\tagform@{\hyperref[#1]{\ref*{#1}#2}}}}
\makeatother

\title{Energy minimization for skyrmions\\ on planar thin films}

\author{Giovanni Di Fratta}
\address{Dipartimento di Matematica e Applicazioni \enquote{R. Caccioppoli}, Università degli Studi di Napoli \enquote{Federico II}, Via Cintia, Complesso Monte S. Angelo, 80126 Napoli, Italy}
\email{Giovanni.DiFratta@unina.it}

\author{Michael Innerberger}
\address{Janelia Research Campus, Howard Hughes Medical Institute, Ashburn, VA, USA}
\email{innerbergerm@hhmi.org \qquad\rm (corresponding author)}

\author{Dirk Praetorius}
\address{TU Wien, Institute for Analysis and Scientific Computing, Wiedner Hauptstr.~8--10/E101/4, 1040 Wien, Austria}
\email{Dirk.Praetorius@asc.tuwien.ac.at}

\author{Valeriy Slastikov}
\address{University of Bristol, School of Mathematics, Bristol BS8 1TW, United Kingdom}
\email{Valeriy.Slastikov@bristol.ac.uk }

\thanks{G.~Di~F., M.~I., and D.~P.\ acknowledge support through the Austrian Science Fund (FWF) through the doctoral school \emph{Dissipation and dispersion in nonlinear PDEs} (grant W1245) and the special research program \emph{Taming complexity in partial differential systems} (grant SFB F65).
M.~I.\ was supported by HHMI Janelia.
G.~Di~F. and V.~S. would also like to thank the Max Planck Institute for Mathematics in the Sciences in Leipzig, the Erwin Schrödinger International Institute for Mathematics and Physics in Wien, and the Vienna Center for PDEs at TU Wien for support and hospitality.
\noindent \textsc{G.~Di F.} is a member of the \emph{Gruppo Nazionale per l’Analisi Matematica, la Probabilità e le loro Applicazioni} (GNAMPA), which is part of the \emph{Istituto Nazionale di Alta Matematica} (INdAM). The work of \textsc{G.~Di F.} was partially supported by the Italian Ministry of Education and Research through the PRIN2022 project \emph{Variational Analysis of Complex Systems in Material Science, Physics and Biology} (No.~2022HKBF5C).  
V.~S. acknowledges support by Leverhulme grant RPG-2018-438.}

\begin{document}

\maketitle

\begin{abstract}
	We consider an energy functional that arises in micromagnetic and liquid crystal theory on thin films.
	In particular, our energy comprises a non-convex term that models anti-symmetric exchange as well as an anisotropy term.
	We devise an algorithm for energy minimization in the continuous case and show weak convergence of a subsequence towards a solution of the corresponding Euler--Lagrange equation.
	Furthermore, an algorithm for numerical energy minimization is presented.
	We show empirically that this numerical algorithm converges to the correct solutions for a benchmark problem without the need for user-supplied parameters, and present a rigorous convergence analysis for important special cases.
\end{abstract}



\section{Introduction}\label{sec:intro}

\subsection{Model problem}

We consider a two-dimensional bounded domain $\Omega \subset \R^2$ with Lipschitz boundary $\partial \Omega$. 
Given $\mu, \kappa \in \R$, we define the energy
\begin{equation}
\label{eq:model-energy}
	\EE(\u)
	:=
	\int_\Omega
	\frac{1}{2} \, |\nabla \u|^2
	+ \kappa \, \u \cdot \curl \u
	+ \frac{\mu}{2} \, |\u \cdot \e_3|^2
	\d{x}
    \quad \text{for all }  u \in \vec{H}^1(\Omega) := [H^1(\Omega)]^3,
\end{equation}
where $\e_i \in \R^3$ is the $i$-th canonical basis vector and the 2D (vector) $\curl$-operator is defined by
\begin{equation}
\label{eq:curl}
	\curl \u
	:=
	\sum_{i=1}^2 \e_i \times \partial_i \u.
\end{equation}
We aim to solve the following minimization problem:
\begin{equation}\label{eq:minimization:E}
    \text{Find}\quad \u \in \MM \coloneqq \set[\big]{\v \in \vec{H}^1(\Omega)}{|\u| = 1 \text{ a.e.\ in } \Omega}
    \quad \text{such that} \quad
    \EE(\u) = \min_{\v \in \MM} \EE(\v).
\end{equation}
The task of solving such minimization problems arises in the theories of micromagnetics and nematic liquid crystals on thin (planar) films~\cite{ddfpr2022}.

The seminal work~\cite{alouges97} introduced a constructive algorithm for computing minimizers to the three-dimensional Dirichlet energy with Dirichlet boundary conditions, which is a simplified variant of~\eqref{eq:model-energy}.
The work~\cite{bartels05} provided and analyzed a finite element version of this algorithm, which indeed involves some subtle technicalities to carry over the argument from the continuous case.
In particular, the algorithm poses a certain angle condition to the underlying triangulation; see Definition~\ref{def:angle-condition} below.

\subsection{Contributions of the present work}

First, we present a surrogate energy functional which has the same set of minimizers as~\eqref{eq:minimization-problem} while consisting only of positive contributions, regardless of the involved parameters $\kappa$ and $\mu$.
This makes the surrogate energy much more convenient for analysis and numerical algorithms.
In particular, the surrogate energy is subsequently used to formulate an iterative algorithm in the spirit of~\cite{alouges97} for computing minimizers of~\eqref{eq:minimization-problem}.
Our algorithm provides a constructive way of finding a minimizing sequence (up to a subsequence) with monotonically decreasing energy starting from an arbitrary initial function.
In analogy to~\cite{alouges97}, we show convergence of this algorithm to solutions of the corresponding Euler--Lagrange equations.
We stress that our problem deals with a model that is considerably more general than that of~\cite{alouges97}, as it incorporates terms of lower order than the Dirichlet energy; in particular, the antisymmetric exchange term $\u \cdot \curl \u$.
This poses several challenges, which we overcome by means of the surrogate energy.
We note that our approach can cover an even larger class of models than~\eqref{eq:model-energy} by permitting different kinds of anisotropy terms; see Section~\ref{subsec:relationship}.
Furthermore, we present a numerical algorithm in the spirit of~\cite{bartels05}, which provides a convenient way of computing approximations to minimizers of $\EE$.
We give implementation details overcoming difficulties that arise from considering natural boundary conditions and that substantially enhance the algorithm given in~\cite{bartels05}.

\subsection{Outline}

In Section~\ref{sec:energy}, we introduce a surrogate energy for~\eqref{eq:model-energy} that is easier to work with on the analytical as well as on the numerical level.
We proceed by stating our energy minimization algorithm in Section~\ref{sec:results}, as well as a convergence result.
Finally, we present a practical finite element discretization in Section~\ref{sec:details} together with some implementation details, numerical experiments, and a convergence analysis for important special cases.

\subsection{Notation}

Throughout, we use boldface letters to denote vectors and italic boldface letters to denote vector-valued functions, whereas for components of vectors and functions we do not (e.g., $u_3$ is the third component of $\constvec{u} \in \R^3$ or some function $\u \colon \Omega \to \R^3$, respectively).
This is to visually distinguish components and collections of vectors.
Moreover, to ensure uniformity of notation even in the discrete setting of Theorem~\ref{th:convergence-discrete}, sequences of vectors generated by algorithms are denoted by upper indices, e.g., $(\constvec{u}^n)_{n \in \N}$.
Finally, the cross product $\constvec{A} \times \constvec{u}$ of a matrix $\constvec{A} \in \R^{3 \times 3}$ and a vector $\constvec{u} \in \R^3$ is taken columnwise.


\section{Surrogate energy \texorpdfstring{$\JJ$}{J}}\label{sec:energy}

\subsection{Setup}

To find minimizers of $\EE$ from~\eqref{eq:model-energy}, we seek minimizers of a surrogate energy with nicer properties for both analysis and numerics.
To this end, we introduce the two-dimensional \emph{helical derivatives} (cf.~\cite{melcher14,dfip20})
\begin{equation}
\label{eq:helical-derivatives}
	\partial_i^\helical \v
	:=
	\partial_i \v + \kappa \, \v \times \e_i,
	\qquad
	\nabla_\helical \v
	:=
	(\partial_1^\helical \v, \partial_2^\helical \v),
	\qquad
	\Delta_\helical \v
	:=
	\sum_{i=1}^{2} \partial_i^\helical \partial_i^\helical \v,
\end{equation}
and define, with $\gamma \in \R$, the surrogate energy
\begin{equation}
\label{eq:energy}
	\JJ(\v)
	:=
	\frac{1}{2} \int_\Omega
	\big[ |\nabla_\helical \v|^2
	+ g_\gamma (\v, \v) \big]
	\d{x}
    \quad \text{for all } \v \in \vec{H}^1(\Omega),
\end{equation}
where the bilinear form $g_\gamma \colon \R^3 \times \R^3 \to \R$ reads
\begin{equation}
\label{eq:gamma-def}
	g_\gamma (\constvec{w}, \constvec{v})
	:=
	\begin{cases}
		\gamma \, (\constvec{w} \cdot \e_3)(\constvec{v} \cdot \e_3) & \gamma \geq 0,\\
		|\gamma| \, (\constvec{w} \times \e_3) \cdot (\constvec{v} \times \e_3) & \gamma < 0.
	\end{cases}
\end{equation}
Note that all terms in $\JJ$ are positive (which is required by our analysis below) as opposed to the terms in $\EE$ from~\eqref{eq:model-energy}.
With $\MM$ from~\eqref{eq:minimization:E}, the surrogate minimization problem then reads as follows:
\begin{equation}
\label{eq:minimization-problem}
	\text{Find}\quad \u \in \MM 
    \quad \text{such that}
	\quad
	\JJ(\u) 
	=
	\min_{\v \in \MM} \JJ(\v).
\end{equation}

\subsection{Relationship of \texorpdfstring{$\JJ$}{J} and \texorpdfstring{$\EE$}{E}}\label{subsec:relationship}

It is straightforward to verify, using the direct method of the calculus of variations, that the minimization problems for both $\EE$ and $\JJ$ over $\MM$ admit solutions.
Moreover, as we now demonstrate, the two problems are variationally equivalent: for suitable choices of the parameters, minimizers of $\JJ$ coincide with that of $\EE$.

\begin{proposition}
	For $\gamma = \mu - \kappa^2$, it holds that
	\begin{equation*}
		\set[\big]{\u \in \MM}{\EE(\u) = \inf_{\v \in \MM} \EE(\v)}
		=
		\set[\big]{\u \in \MM}{\JJ(\u) = \inf_{\v \in \MM} \JJ(\v)}.
	\end{equation*}
\end{proposition}

\begin{proof}
	We begin by examining the relation between the terms $|\v \cdot \e_3|^2$ and $|\v \times \e_3|^2$.
	For any $\v \in \vec{H}^1(\Omega)$, it holds that
	\begin{equation}
	\label{eq:cross-dot-identity}
		|\v \times \e_j|^2
		=
		\sum_{i = 1, \, i \neq j}^3 |\v \cdot \e_i|^2,
	\end{equation}
	which can be seen by explicit computation.	
	From the definition~\eqref{eq:gamma-def} of $g_\gamma$, we see that
	\begin{align}\label{eq:9+}
    \begin{split}
		\gamma \geq 0:
		\quad
		&g_\gamma(\v,\v)
		=
		\gamma \, |\v \cdot \e_3|^2,\\
		\gamma < 0:
		\quad
		&g_\gamma(\v,\v)
		=
		- \gamma \, |\v \times \e_3|^2
		\eqreff{eq:cross-dot-identity}{=}
		\gamma \, |\v \cdot \e_3|^2 - \gamma \, |\v|^2 .
    \end{split}
	\end{align}
	For all $\gamma \in \R$, this yields that $g_\gamma(\v,\v) = \gamma \, |\v \cdot \e_3|^2 + C \, |\v|^2$ with $C = | \min\{0, \gamma\} | \ge 0$.
	Finally, we follow~\cite{dfip20} and expand the helical term in $\JJ$ to see that 
	\begin{equation}
	\label{eq:helical-relationship}
	\begin{split}
		| \nabla_\helical \v |^2
		&\eqreff*{eq:helical-derivatives}{=}
		| \nabla \v |^2 + 2 \kappa \, \v \cdot \curl \v + \kappa^2 \sum_{i=1}^2 | \v \times \e_i |^2\\
		&\eqreff*{eq:cross-dot-identity}{=}
		| \nabla \v |^2 + 2 \kappa \, \v \cdot \curl \v + \kappa^2 (|\v|^2 + | \v \cdot \e_3 |^2).
	\end{split}
	\end{equation}
	Combining the identities above and recalling that $\mu = \kappa^2 + \gamma$, we arrive at
	\begin{align*}
		\JJ(\v)
		&=
		\frac{1}{2} \int_\Omega
			\big[ |\nabla_\helical \v|^2
			+ g_\gamma (\v, \v) \big]
		\d{x}\\
		&\eqreff*{eq:helical-relationship}{=}
		\int_\Omega
			\Big[ \frac{1}{2} |\nabla \v|^2
			+ \kappa \, \v \cdot \curl \v
			+ \frac{\kappa^2 + \gamma}{2}| \v \cdot \e_3 |^2
			+ \frac{\kappa^2 + C}{2} |\v|^2 \Big]
		\d{x}\\
		&=
		\EE(\v) + \frac{\kappa^2 + C}{2} \norm{\v}_{\vec{L}^2(\Omega)}^2.
	\end{align*}
	For $\v \in \MM$, we have that $|\v| = 1$ a.e.\ in $\Omega$ and hence $\norm{\v}_{\vec{L}^2(\Omega)}^2 = |\Omega|$.
	Thus, $\JJ(\v) = \EE(\v) + \mathrm{const}$ for all $\v \in \MM$ so that the sets of minimizers of $\JJ$ and $\EE$ coincide.
\end{proof}

\begin{remark}
	It is also of practical interest to consider anisotropy that favors out-of-plane configurations instead of in-plane ones, which is modeled by the modified energy functional
	\begin{equation*}
		\widetilde{\EE}(\u)
		:=
		\int_\Omega
			\Big[ \frac{1}{2} \, |\nabla \u|^2
			+ \kappa \, \u \cdot \curl \u
			+ \frac{\mu}{2} \, |\u \times \e_3|^2 \Big]
		\d{x}.
	\end{equation*}
	If $\gamma = -(\mu + \kappa^2)$, in analogy to the above computations, we see that
	\begin{equation*}
		\set[\big]{\u \in \MM}{\widetilde{\EE}(\u) = \inf_{\v \in \MM} \widetilde{\EE}(\v)}
		=
		\set[\big]{\u \in \MM}{\JJ(\u) = \inf_{\v \in \MM} \JJ(\v)}.
	\end{equation*}
	Note that, thanks to~\eqref{eq:cross-dot-identity} and the unit-length constraint, any mixture of out-of-plane and in-plane anisotropies is encompassed by our $g_\gamma$ term for an appropriate choice of $\gamma$. Such combinations arise, for instance, in the thin-film limit of three-dimensional models with a stray-field contribution; see~\cite{DiFratta2019var,DiFratta2020,ddfpr2022}.
\end{remark}

\subsection{Euler--Lagrange equations of \texorpdfstring{$\JJ$}{J}}

For fixed $\v \in \vec{H}^1(\Omega)$, we denote by $\vec{g}_\gamma(\v) \in \vec{L}^2(\Omega)$ the Riesz representative of the linear form induced by $g_\gamma(\v, \cdot)$ from~\eqref{eq:gamma-def}, i.e.,
\begin{equation*}
  \int_\Omega \vec{g}_\gamma(\v) \cdot \w \d{x}
  = 
  \int_\Omega g_\gamma(\v, \w) \d{x}
  \quad \text{with} \quad
	\vec{g}_\gamma(\v)
	:=
	\begin{cases}
		\gamma \, v_3 \e_3 & \gamma \geq 0,\\
		|\gamma| \, (v_1 \e_1 + v_2 \e_2) & \gamma < 0.
	\end{cases}
\end{equation*}

\begin{lemma}\label{lemma:el-equations}
	In strong form, the Euler--Lagrange equations associated with the minimization problem for $\JJ$ in \eqref{eq:minimization-problem}  are
	\begin{equation}
	\label{eq:el-strong}
		-\Delta_\helical \u + \vec{g}_\gamma(\u)
		= \big( |\nabla_\helical \u|^2 + g_\gamma(\u, \u) \big) \u
		\quad
		\text{in } \Omega,
	\end{equation}
     subject to the pointwise constraint $|\u| = 1$ in $\Omega$ and the natural boundary condition
    \begin{equation}
	\label{eq:el-strongbc}
		\partial_{\n} \u
		= \kappa \n \times \u,\quad \text{on } \partial\Omega,
	\end{equation}
    where $\n$ denotes the unit outward normal on $\partial\Omega$.
	
    Moreover, a function $\u \in \MM$ is a weak solution of~\eqref{eq:el-strong} if and only if it satisfies the variational identity
	\begin{equation}
	\label{eq:el-equations}
		\int_\Omega \big[ (\u \times \nabla_\helical \u) : \nabla_\helical \v
		+ g_\gamma(\u, \v \times \u) \big] \d{x}
		=
		0
		\quad
		\text{for all } \v \in \vec{H}^1(\Omega) \cap \vec{L}^\infty(\Omega).
	\end{equation}
\end{lemma}

\begin{proof}[Sketch of proof]
	The argument parallels the harmonic-map case.
	First, consider an appropriate variation of the energy to derive the strong Euler–Lagrange equations \eqref{eq:el-strong}.
	Next, insert the vector field $\v \times \u$ as a test function in the weak formulation and employ standard vector identities to show that any weak solution of the Euler–Lagrange equations satisfies \eqref{eq:el-equations}.
	The converse implication is obtained by the fundamental lemma of the calculus of variations.
\end{proof}


\section{Energy minimization}\label{sec:results}

\subsection{Computation of minimizers}

Finding (numerical) solutions of \eqref{eq:minimization-problem} is challenging for two main reasons.
First, the pointwise constraint $|\u(x)|=1$ a.e.\ in $\Omega$ is non-convex, so standard minimization algorithms designed for convex constraints do not apply.
Second, minimizers may fail to be unique (for instance, when $\kappa=0$ and $\gamma\ge 0$ the energy is invariant under orthogonal transformations that fix the $\mathbf e_3$-axis).
For these reasons, a direct attack on the Euler--Lagrange equations is generally not viable.

To handle the unit-length constraint, we use the nearest-point projection onto the unit sphere $\mathbb{S}^2$, i.e.,
\begin{equation}
\label{eq:projection-continuous}
	\Pi \colon
	\MM^+ \to \MM, \,
	\Pi \v \coloneqq \frac{\v}{|\v|},
    \quad \text{where } \MM^+ \coloneqq \set[\big]{\v \in \vec{H}^1(\Omega)}{|\v| \geq 1 \text{ a.e.\ in } \Omega}.
\end{equation}
For variational and numerical purposes, it is convenient to denote by $\KK[\u]$ the tangent space to a given function $\u \in \MM$:
\begin{equation}
\label{eq:tanget-space}
	\KK[\u]
	:=
	\set{\v \in \vec{H}^1(\Omega)}{\v \cdot \u = 0 \text{ a.e.\ in } \Omega}.
\end{equation}
While $\MM$ is non-convex (though weakly closed) in $\vec H^1(\Omega)$, $\KK[\u]$ is a closed linear subspace of $\vec H^1(\Omega)$ and hence a Hilbert space for each $\u \in \MM$; this makes $\KK[\u]$ a natural space for linearized computations around $\u$.


Finally, we introduce the symmetric bilinear form
\begin{equation}
\label{eq:bilinear-form}
	a(\w, \v)
	:=
	\int_\Omega \bigl[ \nabla_\helical \w : \nabla_\helical \v
	+ g_\gamma(\w, \v) \bigr] \d{x}
	\quad
	\text{for all } \w, \v \in \vec{H}^1(\Omega).
\end{equation}
With these ingredients (the projection $\Pi$, the tangent spaces $\KK[\u]$, and the bilinear form $a(\cdot,\cdot)$) we adapt and generalize the numerical scheme from \cite{alouges97} to compute minimizers of \eqref{eq:energy}.

\begin{algorithm}[Continuous energy minimization]\label{alg:alouges}
	\textbf{Input:} Initial guess $\u^{0} \in \MM$.\\
	\textbf{Loop:} For all $n = 0, 1, 2, \ldots$, do
	\begin{itemize}
		\item[{\rm (i)}] Find $\w^{n} \in \KK[\u^{n}]$ such that
		\begin{equation}
		\label{eq:weakELbilinear}
			a(\w^{n}, \v)
			=
			a(\u^{n}, \v)
			\qquad \text{ for all } \v \in \KK[\u^{n}].
		\end{equation}
		
		\item[{\rm (ii)}] Set $\u^{n+1} := \Pi(\u^{n} - \w^{n})$.
	\end{itemize}
	\textbf{Output:} Sequence of functions $(\u^{n})_{n \in \N} \subset \MM$.
\end{algorithm}

%
Our main result states well-posedness and convergence of Algorithm~\ref{alg:alouges} to a critical point of $\JJ$.
The proof is postponed to Subsection~\ref{sec:convergence} below.

\begin{theorem}\label{th:convergence}
	Algorithm~\ref{alg:alouges} is well-defined, i.e., for any $\u^n \in \MM$, the variational formulation~\eqref{eq:weakELbilinear} admits at least one solution $\w^n \in \KK[\u^n]$ and $\u^n - \w^n \in \MM^+$.
	For any sequence $(\u^{n})_{n \in \N}$ generated by Algorithm~\ref{alg:alouges}, there exists a subsequence (not relabeled) and a function $\u \in \MM$ such that, along this subsequence,
	\begin{equation}
	\label{eq:convergence}
		\u^{n} \weak \u
		\quad
		\text{weakly in } \vec{H}^1(\Omega)
        \text{ as } n \to \infty
	\end{equation}
	and $\u$ is a solution of the Euler--Lagrange equations~\eqref{eq:el-strong} of $\JJ$.
\end{theorem}

%
While the existence of a solution $\w^n \in \KK[\u^n]$ to~\eqref{eq:weakELbilinear} is ensured by Theorem~\ref{th:convergence}, uniqueness might fail in general. Clearly,
uniqueness of the solution to~\eqref{eq:weakELbilinear} is guaranteed when Dirichlet boundary conditions are imposed on $\partial\Omega$ (or at least on a nontrivial portion of $\partial \Omega$), as considered in~\cite{alouges97,bartels05}.

As our analysis is concerned with natural boundary conditions (cf.~\eqref{eq:el-strongbc}), the uniqueness of step~{\rm(i)} in Algorithm~\ref{alg:alouges} cannot be taken for granted.
In fact, for certain parameter regimes, additional constraints on $\w^n$ are required to secure uniqueness.
The precise circumstances under which non-uniqueness may occur are characterized in point~{\rm(a)} of the proposition below.

\begin{proposition}\label{prop:w-uniqueness}
	Suppose that $\Omega$ is connected.
	Let $\u \in \MM$ and $\w_1, \w_2 \in \KK[\u]$ be two solutions of the Euler--Lagrange equations~\eqref{eq:weakELbilinear}.
	Then, there hold the following assertions:
	\begin{enumerate}
		\item[\rm(a)] For $\kappa = 0$, the difference $\w_1 - \w_2 \in \KK[\u]$ is constant and satisfies
		\begin{itemize}
			\item $\w_1 - \w_2 \perp \e_3$ a.e. in $\Omega$ if $\gamma >0$,
			\item $\w_1 - \w_2 \in \mathrm{span}\{\e_3\}$ a.e. in $\Omega$ if $\gamma <0$.
		\end{itemize}
		
		\item[\rm(b)] For $\kappa \neq 0$, i.e., in the presence of the anti-symmetric exchange term, the difference $\w_1 - \w_2$ vanishes.
		In particular, the solution of~\eqref{eq:weakELbilinear} is unique.
	\end{enumerate}
\end{proposition}

\begin{proof}
	Since $\KK[\u]$ is a linear space, there holds $\w_1 - \w_2 \in \KK[\u]$.
	By assumption,
	\begin{equation*}
		a(\w_1 - \w_2, \v)
		\eqreff{eq:weakELbilinear}{=}
		0
		\qquad \text{for all }
		\v \in \KK[\u].
	\end{equation*}
	For $\v := \w_1 - \w_2$, we obtain
	\begin{equation}
	\label{eq:Lemmauniquptoconst}
		0 = a(\v, \v)
		=
		\int_\Omega \big[ |\nabla_\helical \v|^2 + g_\gamma(\v, \v) \big] \d{x}.
	\end{equation}
	Non-negativity of $g_\gamma(\v, \v)$ implies $\nabla_\helical \v = 0 = g_\gamma(\v, \v)$ a.e.\ in $\Omega$. In particular, we have
	\begin{equation}
	\label{eq:gioEL}
		\partial_i \v + \kappa \, \v \times \e_i = 0
		\qquad \text{for }
		i = 1,2.
	\end{equation}
	Taking the dot product of~\eqref{eq:gioEL} with $\e_i$ for $i = 1,2,3$, we obtain that
	\begin{equation}
	\label{eq:eq4v1}
	\begin{split}
		&
		\mathllap{\mathrm{(a)}}~~\partial_1 v_1 \phantom{{}- \kappa \, \v \cdot \e_3} = 0,
		\qquad
		\mathllap{\mathrm{(c)}}~~\partial_1 v_2 + \kappa \, \v \cdot \e_3 = 0,
		\qquad
		\mathllap{\mathrm{(e)}}~~\partial_1 v_3 - \kappa \, \v \cdot \e_2 = 0,\\
		&
		\mathllap{\mathrm{(b)}}~~\partial_2 v_1 - \kappa \, \v \cdot \e_3 = 0,
		\qquad
		\mathllap{\mathrm{(d)}}~~\partial_2 v_2 \phantom{{}+ \kappa \, \v \cdot \e_3} = 0,
		\qquad
		\mathllap{\mathrm{(f)}}~~\partial_2 v_3 + \kappa \, \v \cdot \e_1 = 0.
	\end{split}
	\end{equation}
	Next, we now consider the assertions {\rm (a)--(b)} separately and exploit $g_\gamma(\v, \v) = 0$.
	Recall that $g_\gamma(\v,\v) \simeq |\v \cdot \e_3|^2$ for $\gamma > 0$ and $g_\gamma(\v,\v) \simeq |\v \times \e_3|^2$ for $\gamma < 0$.
	
	{\rm\bf (a)}: For $\kappa = 0$, \eqref{eq:eq4v1} simplifies to $\nabla \v = \vec{0}$; hence, $\v$ is constant.
	For $\gamma > 0$, we have that $v_3 = 0$ a.e.\ in $\Omega$ and hence $\v \perp \e_3$.
	For $\gamma < 0$, we have $v_1 = 0 = v_2$ a.e.\ in $\Omega$ and hence $\v \in \mathrm{span}\{\e_3\}$.
	
	{\rm\bf (b)}: For $\kappa \neq 0$, we examine the sign of $\gamma$.
	For $\gamma > 0$, there holds $v_3 = 0$.
	Column~\subref{eq:eq4v1}{e--f} implies that also $v_1 = v_2 = 0$, and hence $\v = 0$.
	For $\gamma <0$, we have $v_1 = 0 = v_2$.
	By~\subref{eq:eq4v1}{a--d}, there follows $v_3 = 0$ and hence $\v = 0$.
	For $\gamma = 0$, the relations $\partial_1 v_1 = \partial_2 v_2 = 0$ prove that $v_1(x_1,x_2) = v_1(x_2)$ and $v_2(x_1,x_2) = v_2(x_1)$, i.e., $v_1$ is independent of $x_1$ while $v_2$ is independent of $x_2$.
	Adding~\subref{eq:eq4v1}{b--c}, we see $\partial_1 v_2(x_1) + \partial_2 v_1(x_2)=0$ for every $(x_1,x_2)\in\Omega$.
	This implies that $\partial_1 v_2(x_1)$ and $\partial_2 v_1(x_2)$ are both constant.
	Hence, \subref{eq:eq4v1}{b} yields that $\kappa \, v_3 = \partial_2 v_1(x_2)$ is constant.
	From $\kappa \neq 0$, we thus obtain that $v_3$ is constant in $\Omega$.
	Column~\subref{eq:eq4v1}{e--f} then yields that both $v_1 = 0 = v_2$.
	From~\subref{eq:eq4v1}{b}, we finally conclude that also $v_3$ is zero and hence $\v = 0$ also in this case.
\end{proof}


\subsection{Proof of Theorem~\ref{th:convergence}}\label{sec:convergence}

First, we ensure that both steps of Algorithm~\ref{alg:alouges} lead to a decrease in the energy $\JJ$.
We start by showing that subtracting the tangential update obtained in step (i) decreases the energy (although this intermediate update lies in $\MM^+$ from~\eqref{eq:projection-continuous} rather than in $\MM$ from~\eqref{eq:minimization:E}).
	
\begin{lemma}\label{lemma:continuous-energy-decrease}
	Let $\u \in \vec{H}^1(\Omega)$ and $\w \in \KK[\u]$ such that $a(\w, \v) = a(\u, \v)$ for all $\v \in \KK[\u]$.
	Then, it holds that
	\begin{equation}
	\label{eq:continuous-energy-decrease}
		\JJ(\u - \w)
		=
		\JJ(\u) - \JJ(\w) \le \JJ(\u).
	\end{equation}
\end{lemma}

\begin{proof}
	We choose $\v = 2 \w \in \KK[\u]$ to obtain that
	\begin{align*}
		0
		=
		a(\u - \w, 2 \w)
		&=
		a(\u, \u) - a(\u, \u) + 2 a(\u, \w) - 2 a(\w, \w)\\
		&=
		a(\u, \u) - a(\u - \w, \u - \w) - a(\w, \w)
        \\&
        = 2 \JJ(\u) - 2 \JJ(\u-\w) - 2 \JJ(\w).
	\end{align*}
	By non-negativity of $\JJ$, it follows that
	\begin{equation*}
		\JJ(\u - \w)
		=
		\JJ(\u) - \JJ(\w)
        \le \JJ(\u).
	\end{equation*}
	This concludes the proof.
\end{proof}

Also, the projection step (ii) decreases the energy.

\begin{lemma}\label{lemma:continuous-projection-decrease}
	Let $\u \in \MM^+$. 
	Then, it holds that
	\begin{equation}
	\label{eq:continuous-projection-decrease}
		\JJ(\Pi(\u))
		\leq
		\JJ(\u).
	\end{equation}
\end{lemma}

\begin{proof}
	It follows from the definition~\eqref{eq:projection-continuous} of the projection $\Pi(\u)$ that
	\begin{equation*}
		|\Pi(\u) \cdot \e_i|^2
        = \frac{|\u \cdot \e_i|^2}{|\u|^2}
		\leq
		|\u \cdot \e_i|^2
		\quad \text{for all } i=1,2,3,
	\end{equation*}
    where, here and in the following, all estimates hold a.e.\ in $\Omega$.
	From~\eqref{eq:9+}, we then see that
	\begin{equation}
	\label{eq:projection-inequality-gamma}
		g_\gamma(\Pi(\u), \Pi(\u))
		\leq
		g_\gamma(\u, \u).
	\end{equation}
	Thus, it remains to inspect the helical term in the energy density.
	Note that the derivative of the projection can be computed explicitly to be
	\begin{equation}
	\label{eq:projection-derivative}
		\big( \nabla \Pi \big) (\u)
		=
		\frac{1}{|\u|} \Big(\mathbf{1} - \frac{\u \u^\transposed}{|\u|^2} \Big).
	\end{equation}
	Inserting this into the helical energy density, the chain rule proves that
	\begin{align*}
		|\nabla_\helical \Pi (\u)|^2
		&\eqreff*{eq:helical-derivatives}{=}
		\sum_{i=1}^2 | \big( \nabla \Pi \big) (\u) \partial_i \u + \kappa \, \Pi(\u) \times \e_i|^2\\
		&\eqreff*{eq:projection-derivative}{=}
		\sum_{i=1}^2 \Big| \frac{1}{|\u|} \partial_i \u - \frac{1}{|\u|^3} \u \u^\transposed \partial_i \u +  \frac{\kappa}{|\u|} \u \times \e_i \Big|^2\\
		&=
		\sum_{i=1}^2 \Big| \frac{1}{|\u|} (\partial_i \u + \kappa \, \u \times \e_i) - \frac{1}{|\u|^3} (\u \cdot \partial_i \u) \u \Big|^2\\
		&=
		\sum_{i=1}^2 \frac{1}{|\u|^2} | \partial_i \u + \kappa \, \u \times \e_i |^2
		+ \frac{1}{|\u|^4} | \u \cdot \partial_i \u |^2	- \frac{2}{|\u|^4} (\partial_i \u + \kappa \, \u \times \e_i) \cdot (\u \cdot \partial_i \u) \u.
	\end{align*}
	Note that $\u \times \e_i \perp (\u \cdot \partial_i \u) \u$ and $\partial_i \u \cdot (\u \cdot \partial_i \u) \u = | \u \cdot \partial_i \u |^2$. With $|\u| \geq 1$ a.e.\ in $\Omega$, we thus get
	\begin{align}
	\begin{split}
	\label{eq:projection-inequality-helical}
		|\nabla_\helical \Pi (\u)|^2
		&=
		\sum_{i=1}^2 \frac{1}{|\u|^2} | \partial_i \u + \kappa \, \u \times \e_i |^2
		- \frac{1}{|\u|^4} | \u \cdot \partial_i \u |^2 \\
		&\leq
		\sum_{i=1}^2 | \partial_i \u + \kappa \, \u \times \e_i |^2
		=
		| \nabla_\helical \u |^2.
	\end{split}
	\end{align}
	Integrating~\eqref{eq:projection-inequality-gamma} and~\eqref{eq:projection-inequality-helical} over $\Omega$, we conclude the proof.
\end{proof}

The next lemma provides quantitative two-sided estimates for the surrogate energy $\JJ$. From these estimates, we deduce that $\JJ$ controls the $\vec{H}^1$-seminorm and hence is coercive on $\MM$.
\begin{lemma}
\label{lemma:lower-bound}
	For $\u \in \vec{H}^1(\Omega)$, we have that
	\begin{align}
	\label{eq:lemma-lower-bound}
		\norm{\nabla \u}_{\vec{L}^2(\Omega)}^2
		&\leq
		4 \JJ(\u) + 4 \kappa^2 \, \norm{\u}_{\vec{L}^2(\Omega)}^2,\\
	\label{eq:lemma-upper-bound}
		\JJ(\u)
		&\leq
		\norm{\nabla \u}_{\vec{L}^2(\Omega)}^2 + \Big( \kappa^2 + \frac{|\gamma|}{2} \Big) \, \norm{\u}_{\vec{L}^2(\Omega)}^2.
	\end{align}
\end{lemma}

\begin{proof}
	We first show the estimate~\eqref{eq:lemma-lower-bound}.
	For $\kappa = 0$, the estimate~\eqref{eq:lemma-lower-bound} is obvious by definition of $\JJ$ (and holds even with a factor $2$ instead of $4$).
	Let $\kappa \neq 0$.
	The Young inequality $2ab \leq a^2 / 2 + 2 b^2$ shows that
	\begin{align*}
		|\nabla_\helical \u|^2
		&\eqreff*{eq:helical-derivatives}\geq
		\sum_{i=1}^2 \Big( |\partial_i \u|^2 + \kappa^2 |\u \times \e_i|^2
		- 2 |\kappa \, \partial_i \u \cdot (\u \times \e_i)| \Big)\\
		&\geq
		\big(1 - 1/2 \big) |\nabla \u|^2
		+ \big( \kappa^2 - 2 \kappa^2 \big) \sum_{i=1}^2 |\u \times \e_i|^2.
	\end{align*}
	Hence, we get that
	\begin{equation*}
		\frac{1}{2} |\nabla \u|^2
		\leq
		|\nabla_\helical \u|^2 + \kappa^2 \sum_{i=1}^2 |\u \times \e_i|^2
		\leq
		|\nabla_\helical \u|^2 + 2 \kappa^2 |\u|^2.
	\end{equation*}
	Adding the term $g_\gamma(\u,\u) \geq 0$ on the right-hand side and integrating over $\Omega$, this shows~\eqref{eq:lemma-lower-bound}.
	The estimate~\eqref{eq:lemma-upper-bound} follows from $g_\gamma(\u,\u) \leq |\gamma| \, \norm{\u}_{\vec{L}^2(\Omega)}^2$ as well as
	\begin{equation*}
		\norm{\nabla_\helical \u}_{\vec{L}^2(\Omega)}^2
		\leq
		2 \big( \norm{\nabla \u}_{\vec{L}^2(\Omega)}^2 + \kappa^2 \, \norm{\u}_{\vec{L}^2(\Omega)}^2 \big).
	\end{equation*}
	Combining these estimates, we conclude the proof.
\end{proof}

Finally, we can prove the convergence of Algorithm~\ref{alg:alouges}.

\begin{proof}[\bfseries Proof of Theorem~\ref{th:convergence}]
The proof is split into five steps.\smallskip

	\textbf{Step 1 (Algorithm~\ref{alg:alouges} is well-defined):} We note that  $\JJ(\u) \geq 0$ is convex, continuous, and coercive.
	By convexity of $\JJ$, \eqref{eq:weakELbilinear} is equivalent to finding minimizers of $\JJ(\u^{n} - (\cdot))$ over $\KK[\u^n]$ by means of the Euler--Lagrange equations of the corresponding energy.
	Since $\JJ$ is convex and continuous, it is also sequentially lower semicontinuous, and $\KK[\u^n]$ is a reflexive Banach space.
	Therefore, the existence of at least one solution to~\eqref{eq:weakELbilinear} follows by the direct method of calculus of variations, which means that step~(i) in Algorithm~\ref{alg:alouges} is well-defined.
	For all $n \in \N$ there holds pointwise orthogonality $\u^{n} \perp \w^{n}$ and thus
	\begin{equation*}
		|\u^{n} - \w^{n}|^2
		=
		|\u^{n}|^2 + |\w^{n}|^2
		\geq 1
		\quad \text{a.e.\ in } \Omega.
	\end{equation*}
	Therefore, $\u^n - \w^n \in \MM^+$ and step~{\rm (ii)} in Algorithm~\ref{alg:alouges} is always well-defined.\smallskip

	\textbf{Step 2 (Energy decrease along computed sequence):} 
    It is an immediate consequence of Lemma~\ref{lemma:continuous-energy-decrease} and Lemma~\ref{lemma:continuous-projection-decrease} that
	\begin{equation}
	\label{eq:energy-decrease}
		\JJ(\u^{n+1})
		=
		\JJ(\Pi(\u^{n} - \w^{n}))
		\eqreff{eq:continuous-energy-decrease}{\leq}
		\JJ(\u^{n} - \w^{n})
		\eqreff{eq:continuous-projection-decrease}{\leq}
		\JJ(\u^{n})
		\quad
		\text{for all } n \in \N_0.
	\end{equation}
		
	\textbf{Step 3 (Energy of tangential updates vanishes in the limit):} 
	Together with $\u^{n+1} = \Pi(\u^{n} - \w^{n})$, Lemma~\ref{lemma:continuous-energy-decrease} and Lemma~\ref{lemma:continuous-projection-decrease} yield
	\begin{equation*}
		0
		\leq
		\JJ(\w^{n})
		\eqreff{eq:continuous-energy-decrease}{=}
		\JJ(\u^{n}) - \JJ(\u^{n} - \w^{n})
		\eqreff{eq:continuous-projection-decrease}{\leq}
		\JJ(\u^{n}) - \JJ(\u^{n+1}).
	\end{equation*}
	Summing the last inequality over $n = 0, \ldots, N$, we get that
	\begin{equation*}
		0
		\leq
		\sum_{n=0}^N \JJ(\w^{n})
		\leq
		\JJ(\u^{0}) - \JJ(\u^{N+1})
		\leq
		\JJ(\u^{0})
		\qquad \text{for all } N \in \N.
	\end{equation*}
	In particular, this yields that
    \begin{equation}
	\label{eq:w-convergence}
		\JJ(\w^{n}) \to 0
		\quad \text{as } n \to \infty.
	\end{equation}
	
	\textbf{Step 4 (Existence of limit $\u$ of weakly convergent subsequence):} We infer from Lemma~\ref{lemma:lower-bound}, \eqref{eq:energy-decrease}, and the fact that $|\u^{n}| = 1$ a.e.\ in $\Omega$ that  
	\begin{equation}
	\label{eq:norm-bound}
		\norm{\u^{n}}_{\vec{H}^1(\Omega)}^2
		\eqreff{eq:lemma-lower-bound}{\lesssim}
		\JJ(\u^{n}) + \norm{\u^{n}}_{\vec{L}^2(\Omega)}
		\eqreff{eq:energy-decrease}{\leq}
		\JJ(\u^{0}) + |\Omega|.
	\end{equation}
	In particular, $(\u^{n})_{n \in \N}$ is uniformly bounded in $\vec{H}^1(\Omega)$. Hence there exists a subsequence of $(\u^{n})_{n \in \N}$ (which is not relabeled) and a limit function $\u \in \vec{H}^1(\Omega)$ such that $\u^{n} \weak \u$ in $\vec{H}^1(\Omega)$.
	Since $\MM$ is weakly closed in $\vec{H}^1(\Omega)$, we conclude that $\u \in \MM$. \smallskip
	
	\textbf{Step 5 ($\u$ solves Euler--Lagrange equations~(\ref{eq:el-equations}) of $\boldsymbol\JJ$):} Let $\v \in \vec{H}^1(\Omega) \cap \vec{L}^\infty(\Omega)$.
	Note that $\v \times \u^{n} \in \KK[\u^{n}]$.
	The properties of the triple product imply that
	\begin{equation*}
		\nabla_\helical \u^{n} : \nabla_\helical (\v \times \u^{n})
		=
		\u \times \nabla_\helical \u^{n} : \nabla_\helical \v.
	\end{equation*}
	Using this identity, we derive
	\begin{align}
	\nonumber
		&\int_\Omega \big[ (\u^{n} \times \nabla_\helical \u^{n}) : \nabla_\helical \v
		+ g_\gamma(\u^{n}, \v \times \u^{n}) \big] \d{x}
		=
		a(\u^{n}, \v \times \u^{n})
		=
		a(\w^{n}, \v \times \u^{n}) \qquad \\
	\label{eq:convergence-identity}
		&\qquad \leq
		a(\w^{n}, \w^{n})^{1/2} a(\v \times \u^{n}, \v \times \u^{n})^{1/2} \\
	\nonumber
		&\qquad =
		2 \, \JJ(\w^{n})^{1/2} \JJ(\v \times \u^{n})^{1/2} 
		\\ \nonumber
        &\qquad
        \eqreff{eq:lemma-upper-bound}{\lesssim}
		\JJ(\w^{n})^{1/2} \norm{\v \times \u^{n}}_{\vec{H}^1(\Omega)}
        \lesssim \JJ(w^n) \, \norm{u^n}_{\vec{H}^1(\Omega)}
        \eqreff{eq:norm-bound}\lesssim \JJ(w^n).
	\end{align}
	With Step~3, the right-hand side of~\eqref{eq:convergence-identity} vanishes in the limit.
	For the left-hand side, we can exploit weak convergence in $\vec{H}^1(\Omega)$ and strong convergence in $\vec{L}^2(\Omega)$ of $\u^{n}$.
	Overall, \eqref{eq:convergence-identity} converges to
	\begin{equation*}
		\int_\Omega (\u \times \nabla_\helical \u) : \nabla_\helical \v
		+ g_\gamma(\u, \v \times \u) \d{x}
		=
		0
        \quad \text{for all } \v \in \vec{H}^1(\Omega) \cap \vec{L}^\infty(\Omega).
	\end{equation*}
	Together with Lemma~\ref{lemma:el-equations}, this concludes the proof.
\end{proof}


\section{Finite element discretization}\label{sec:details}

\subsection{Algorithm}

In this section, we extend the approach of \cite{bartels05}.
To this end, let $\TT_h$ be a conforming triangulation of $\Omega$ into compact non-degenerate triangles.
Let $\NN_h$ be the set of nodes of $\TT_h$.
We define the lowest-order finite element space
\begin{equation*}
	\SS^1(\TT_h)
	:=
	\set[]{v \in H^1(\Omega)}{v|_T \text{ is affine } \forall T \in \TT_h}.
\end{equation*}
A convenient basis of $\SS^1(\TT_h)$ is the so-called nodal basis $\set[]{\varphi_{\z}}{\z \in \NN_h}$ consisting of hat functions, i.e., $\TT_h$-piecewise affine functions that satisfy $\varphi_{\z}(\z') = \delta_{\z\z'}$ for all $\z,\z' \in \NN_h$.
A vector-valued finite element function $\v_h \in \vec{\SS}^1(\TT_h) := [\SS^1(\TT_h)]^3$ and its Jacobian can then be written as
\begin{equation*}
	\v_h = \sum_{\z \in \NN_h} \v_h(\z) \varphi_{\z}
	\quad\text{and}\quad
	\nabla \v_h = \sum_{\z \in \NN_h} \v_h(\z) \nabla \varphi_{\z},
\end{equation*}
respectively, where the column vectors $\v_h(\z) \in \R^3$ are the values at the nodes $\z \in \NN_h$; note that $\nabla \v_h$ is pointwise in $\R^{3 \times 2}$ with $(\nabla \v_h)_{jk} = \partial_k v_{h,j}$.

For $\u_h \in \vec{\SS^1}(\TT_h)$, the constraint $|\u_h| = 1$ a.e.\ in $\Omega$ can only be satisfied by constant functions.
Therefore, it is natural to prescribe the constraint only at the nodes.
In analogy to Section~\ref{sec:results}, we define $\MM_h := \set{\u_h \in \vec{\SS^1}(\TT_h)}{|\u_h(\z)| = 1 \text{ for all } \z \in \NN_h}$ and $\KK_h[\u_h] := \set{\v_h \in \vec{\SS^1}(\TT_h)}{\v_h(\z) \cdot \u_h(\z) = 0 \text{ for all } \z \in \NN_h}$ as well as the nodal unit-length projection
\begin{align}
\label{eq:projection}
	\Pi_h &\colon
	\MM_h^+ \to \MM_h \colon
	\v_h \mapsto \sum_{\z \in \NN_h} \frac{\v_h(\z)}{|\v_h(\z)|} \varphi_{\z},\\
\nonumber
	&\text{ with }
	\MM_h^+ := \set[\big]{\v_h \in \vec{\SS^1}(\TT_h)}{|\v_h(\z)| \geq 1 \text{ for all } \z \in \NN_h}.
\end{align}
This means that, in general, the unit-length constraint as well as orthogonality cannot be achieved in the interior of triangles.
However, there still holds $|\u_h| \leq 1$ a.e.\ in $\Omega$ for any $\u_h \in \MM_h$.
We define the discrete version of~\eqref{eq:energy} as
\begin{equation}
\begin{split}
\label{eq:energy-discrete}
	\JJ_h(\v_h)
	&:=
	\frac{1}{2} a_h(\v_h, \v_h)
	\quad
	\text{for all } \v_h \in \vec{\SS^1}(\TT_h), \text{ where }\\
	a_h(\v_h, \w_h)
	&:=
	\int_\Omega \Big[ \nabla_\helical \v_h : \nabla_\helical \w_h + \II_h\big( g_\gamma(\v_h, \w_h) \big) \Big] \d{x}
	\quad
	\text{for all } \v_h, \w_h \in \vec{\SS^1}(\TT_h),
\end{split}
\end{equation}
where $\II_h \colon C(\overline{\Omega}) \to \SS^1(\TT_h)$, $\II_h f := \sum_{\z \in \NN_h} f(\z) \varphi_{\z}$ is the nodal interpolation operator.
Then, the discrete minimization problem reads:
\begin{equation}
\label{eq:minimization-discrete}
	\text{Find } \u_h \in \MM_h \text{ such that}
	\quad
	\JJ_h(\u_h)
	=
	\min_{\v_h \in \MM_h} \JJ_h(\v_h).
\end{equation}
The term $\int_\Omega \II_h(g_\gamma(\v_h, \v_h)) \d{x}$ will be referred to as \emph{mass lumping} term.
In analogy to~\cite{bartels05}, the discrete problem~\eqref{eq:minimization-discrete} is solved via a discretized version of Algorithm~\ref{alg:alouges}.

\begin{algorithm}[Discrete energy minimization]\label{alg:alouges-discrete}
	\textbf{Input:} Conforming simplicial triangulation $\TT_h$, initial guess $\u_h^{0} \in \MM_h$, tolerance $\mathrm{tol} > 0$.\\
	\textbf{Loop:} For all $n = 0, 1, 2, \ldots$ do
	\begin{itemize}
		\item[{\rm (i)}] Compute $\w_h^{n} \in \KK_h[\u_h^{n}]$ such that
		\begin{equation}
		\label{eq:discreteELbilinear}
			a_h(\w_h^{n}, \v_h)
			=
			a_h(\u_h^{n}, \v_h)
			\qquad \text{ for all } \v_h \in \KK_h[\u_h^{n}].
		\end{equation}
		
		\item[{\rm (ii)}] If $\JJ_h(\w_h^{n}) \leq \mathrm{tol}$, define $\u_h := \u_h^{n}$ and terminate.
	
		\item[{\rm (iii)}] Otherwise, set $\u_h^{n+1} := \Pi_h(\u_h^{n} - \w_h^{n})$.
	\end{itemize}
	\textbf{Output:} Approximate minimizer $\u_h \in \MM_h$ of $\JJ_h$.
\end{algorithm}

Convergence of Algorithm~\ref{alg:alouges-discrete} requires the following technical assumption on the triangulation.

\begin{definition}[Angle condition]\label{def:angle-condition}
	For any two triangles $T_1, T_2 \in \TT_h$ that share an edge (and, hence, two nodes), let $\alpha_1$ and $\alpha_2$ be the interior angles of the triangles $T_1$ and $T_2$ at the node which they do not share, respectively; see Figure~\ref{fig:angle-condition} for a visualization.
	We say that $\TT_h$ satisfies the angle condition, if there holds
	\begin{equation}
	\label{eq:angle-condition}
		\cot \alpha_1 + \cot \alpha_2 \geq 0
		\quad
		\text{for all neighboring } T_1, T_2 \in \TT_h.
	\end{equation}
	A sufficient condition for~\eqref{eq:angle-condition} to hold is $\alpha_1 + \alpha_2 \leq \pi$.
	In particular, this is clearly satisfied if all angles of $\TT_h$ are non-obtuse.
\end{definition}

\begin{figure}
	\begin{tikzpicture}
	\begin{scope}
		\coordinate (A) at (0,0);
		\coordinate (B) at (2,0);
		\coordinate (C) at (1,2);
		\coordinate (D) at (3,2);
		
		\draw (A) -- (B) -- (C) -- cycle;
		\draw (B) -- (D) -- (C);
		
		\node (M1) at ($2/8*(A)+3/8*(B)+3/8*(C)$) {$T_1$};
		\node (M2) at ($3/8*(B)+2/8*(D)+3/8*(C)$) {$T_2$}; 
		
		\draw pic[draw,angle radius=0.8cm,"$\alpha_1$"] {angle=B--A--C};
		\draw pic[draw,angle radius=0.8cm,"$\alpha_2$"] {angle=C--D--B};
	\end{scope}
		
	\begin{scope}[xshift=0.4\textwidth]
		\coordinate (A) at (0,0);
		\coordinate (B) at (2,0);
		\coordinate (C) at (1,2);
		\coordinate (D) at (3,2);
		
		\draw (A) -- (B) -- (D) -- cycle;
		\draw (A) -- (C) -- (D);
		
		\node (M1) at ($5/8*(A)+2/8*(B)+1/8*(D)$) {$T_1$};
		\node (M2) at ($0.58*(A)+0.24*(C)+0.18*(D)$) {$T_2$}; 
		
		\draw pic[draw,angle radius=0.8cm,"$\alpha_1$"] {angle=D--B--A};
		\draw pic[draw,angle radius=0.8cm,"$\alpha_2$"] {angle=A--C--D};
	\end{scope}
	\end{tikzpicture}
	\caption{Illustration of the angle condition with notation from Definition~\ref{def:angle-condition}.
	Left: configuration satisfying the angle condition.
	Right: configuration \emph{not} satisfying the angle condition.}
	\label{fig:angle-condition}
\end{figure}
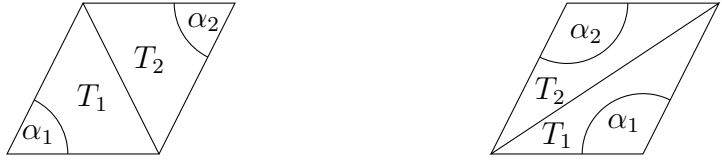

The next theorem states convergence of Algorithm~\ref{alg:alouges-discrete} to a stationary point of $\JJ$ for $\kappa = 0$; see also the discussion in Section~\ref{sec:helical-increase} on the case $\kappa \neq 0$. The proof of Theorem~\ref{th:convergence-discrete} is postponed to Section~\ref{sec:discrete} below.

\begin{theorem}\label{th:convergence-discrete}
	Let $\kappa = 0$. For any $h > 0$, let $\TT_h$ be a conforming simplicial triangulation with maximal mesh-width $h$ that satisfies the angle condition~\eqref{eq:angle-condition}.
	Let further $\mathrm{tol}_h > 0$ be a tolerance with $\mathrm{tol}_h \to 0$ as $h \to 0$. Then, there hold the following statements~{\rm(i)--(ii)}:

	{\rm (i)}
	Algorithm~\ref{alg:alouges-discrete} is well-defined, i.e., for any $\u_h^n \in \MM_h$, the variational problem~\eqref{eq:discreteELbilinear} admits a solution $\w_h^n \in \KK_h[\u_h^n]$ so that $\u_h^n - \w_h^n \in \MM_h^+$. Moreover, the termination criterion $\JJ_h(\w_h^n) \leq \mathrm{tol}_h$ is reached in finite time and Algorithm~\ref{alg:alouges-discrete} provides $\u_h = \u_h^n\in \MM_h$ with $\JJ_h(\u_h) \leq \JJ_h(\u_h^0)$ for some $n \in \N_0$.

	{\rm (ii)}
	Suppose that the initial guesses $\u_h^{0} \in \MM_h$ are uniformly bounded, i.e., there exists $C_0 > 0$ such that $\JJ_h(\u_h^{0}) \leq C_0$ for all $h>0$.
	Let $\u_h$ be the output of Algorithm~\ref{alg:alouges-discrete} with $(\TT_h, \u_h^{0}, \mathrm{tol}_h)$ as input.
	Then, there exists a weak solution $\u \in \MM$ of the Euler--Lagrange equations~\eqref{eq:el-strong} of $\JJ$ and a subsequence of $(\u_h)_h$ such that, along this subsequence,
	\begin{equation}
	\label{eq:convergence-discrete}
		\u_h \weak \u \quad \text{weakly in } \vec{H}^1(\Omega)
		\qquad\text{and}\qquad
		\JJ(\u)
		\leq
		\liminf_{h \to 0} \JJ_h(\u_h).
	\end{equation}
\end{theorem}

\begin{remark}
	Note that the proof of well-definedness of the discrete problem~\eqref{eq:discreteELbilinear} below does not make use of $\kappa = 0$.
	Thus, Algorithm~\ref{alg:alouges-discrete} is well-defined for all $\kappa \in \R$. However, the current proof of energy reduction and hence convergence (resp.\ termination) relies on $\kappa = 0$, even though Algorithm~\ref{alg:alouges-discrete} empirically appears to  converge for any $\kappa \in \R$; see Section~\ref{sec:reduced-model} below for a related practical experiment.
\end{remark}

As for the continuous case in Section~\ref{sec:results}, one can prove that the discrete tangential update $\w_h^n \in \KK[\u_h^n]$ of~\eqref{eq:discreteELbilinear} exists, but can be non-unique in general.
The next proposition comments on the uniqueness of the tangential update in the discrete setting.
It suggests additional linear constraints to add to the discrete system of equations of~\eqref{eq:weakELbilinear} to enforce uniqueness.

\begin{proposition}
	\label{prop:numerical-uniqueness}
	Suppose that $\Omega$ is connected.
	Let $\u_h \in \MM_h$ and $\w_h^1, \w_h^2 \in \KK_h[\u_h]$ be two solutions of the discrete Euler--Lagrange equations~\eqref{eq:discreteELbilinear}, i.e.,
	\begin{equation*}
		a_h(\w_h^1, \v_h)
		=
		a_h(\u_h, \v_h)
		=
		a_h(\w_h^2, \v_h)
		\quad \text{for all } \v_h \in \KK_h[\u_h].
	\end{equation*}
	Then, there hold the following assertions:
	\begin{enumerate}
		\item[\rm(a)] For $\kappa = 0$, the difference $\w_h^1 - \w_h^2 \in \KK_h[\u_h]$ is constant and satisfies
		\begin{itemize}
			\item $\w_h^1(\constvec{z}) - \w_h^2(\constvec{z}) \perp \e_3$ for all $\constvec{z} \in \NN_h$ if $\gamma >0$,
			\item $\w_h^1(\constvec{z}) - \w_h^2(\constvec{z}) \in \mathrm{span}\{\e_3\}$ for all $\constvec{z} \in \NN_h$ if $\gamma <0$.
		\end{itemize}
		
		\item[\rm(b)] For $\kappa \neq 0$, i.e., in the presence of the anti-symmetric exchange term, the difference $\w_h^1 - \w_h^2$ vanishes.
		In particular, the solution of~\eqref{eq:weakELbilinear} is unique.
	\end{enumerate}
	Moreover, let $\kappa = 0$ and let
	\begin{equation*}
		B
		:=
		\begin{cases}
			\set{\u_h(\constvec{z})}{\constvec{z} \in \NN_h} & \text{if } \gamma = 0, \\
			\set{\u_h(\constvec{z})}{\constvec{z} \in \NN_h} \cup \{\e_3\} & \text{if } \gamma > 0, \\
			\set{\u_h(\constvec{z})}{\constvec{z} \in \NN_h} \cup \{\e_1, \e_2\} & \text{if } \gamma < 0.
		\end{cases}
	\end{equation*}
	Let further $B_\perp$ be an orthonormal basis of the orthogonal complement $(\mathrm{span} \, B)^\perp$.
	Let $m$ be the cardinality of $B_\perp = \{\constvec{b}_1, \ldots, \constvec{b}_m\}$.
	Adding the linear constraints $\int_\Omega \w_h \cdot \constvec{b} \d{x} = 0$ for all $\constvec{b} \in B_\perp$ to~\eqref{eq:discreteELbilinear} as Lagrange multipliers yields the following saddle point problem: 
	find $(\w_h, \constvec{\lambda}) \in \KK_h[\u_h] \times \R^m$ such that
	\begin{equation}\label{eq:saddle-point}
	\begin{split}
		a_h(\w_h, \v_h) + b_h(\v_h, \constvec{\lambda})
		&= a_h(\u_h, \v_h) \phantom{0}
		\text{ for all } \v_h \in \KK_h[\u_h], \\
		b_h(\w_h, \constvec{\mu}) \phantom{{}+ b_h(\v_h, \constvec{\lambda})}
		&= 0 \phantom{a_h(\u_h, \v_h)} \text{ for all } \constvec{\mu} \in \R^m
	\end{split}
	\end{equation}
	with the bilinear form $b_h \colon \KK_h[\u_h] \times \R^m \to \R$ defined by
	\begin{equation*}
		b_h(\v_h, \constvec{\mu})
		:=
		\sum_{i=1}^m \mu_i \int_\Omega \v_h \cdot \constvec{b}_i \d{x}
		\quad
		\text{for all } \v_h \in \KK_h[\u_h], \constvec{\mu} \in \R^m.
	\end{equation*}
	The saddle point problem~\eqref{eq:saddle-point} admits a unique solution $(\w_h, \constvec{\lambda}) \in \KK_h[\u_h] \times \R^m$ and $\w_h$ solves~\eqref{eq:discreteELbilinear}.
\end{proposition}

\begin{proof}
	The statements (a)--(b) follow analogously to the proof of Proposition~\ref{prop:w-uniqueness}.
	To show that the saddle point problem~\eqref{eq:saddle-point} admits a unique solution, we employ the Brezzi theorem~\cite{bbf2013}, i.e., we show that the operator $\BB \colon \KK_h[\u_h] \to (\R^m)^* : \v_h  \mapsto b_h(\v_h, \cdot)$ is surjective and that $a_h(\cdot, \cdot)$ is coercive on the kernel of $\BB$.

	To show surjectivity, it suffices to show that, for any $\constvec{\mu} \in \R^m \setminus \{\constvec{0}\}$, there exists $\v_h \in \KK_h[\u_h]$ such that $b_h(\v_h, \constvec{\mu}) \neq 0$.
	Given $\constvec{\mu} \in \R^m \setminus \{\constvec{0}\}$, we define the constant function $\v_h := \sum_{i=1}^m \mu_i \constvec{b}_i$.
	From the definition of $B_\perp$, we obtain that $\v_h \in \KK_h[\u_h]$ and
	\begin{equation*}
		b_h(\v_h, \constvec{\mu})
		=
		\sum_{i=1}^m \mu_i \int_\Omega \Bigl( \sum_{j=1}^m \mu_j \constvec{b}_j \Bigr) \cdot \constvec{b}_i \d{x}
		=
		\sum_{i=1}^m \mu_i^2 \int_\Omega 1 \d{x}
		=
		|\constvec{\mu}|^2 |\Omega| > 0.
	\end{equation*}

	To show coercivity of $a_h(\cdot, \cdot)$ on the kernel of $\BB$, let $\v_h \in \KK_h[\u_h]$ with $\BB\v_h = \constvec{0}$ and note that $a(\v_h, \v_h) = 0$ implies that $\v_h \equiv \constvec{c} \in \R^3$ is constant as well as $\constvec{c} \perp \mathrm{span} \, B$ according to the statements (a)--(b) of the proposition.
	Since $|\Omega| \constvec{c} \cdot \constvec{b} = \int_\Omega \v_h \cdot \constvec{b} \d{x} = b_h(\v_h, \constvec{b}) = 0$ for all $\constvec{b} \in B_\perp$ due to $\BB\constvec{c} = \BB\v_h = \constvec{0}$, we obtain that also $\constvec{c} \perp \mathrm{span} \, B_\perp = (\mathrm{span} \, B)^\perp$ and, hence, $\constvec{c} = \constvec{0}$.

	Thus, the assumptions of the Brezzi theorem are satisfied and the saddle point problem~\eqref{eq:saddle-point} admits a unique solution $(\w_h, \constvec{\lambda}) \in \KK_h[\u_h] \times \R^m$.
	To show that $\w_h$ solves~\eqref{eq:discreteELbilinear}, we test~\eqref{eq:saddle-point} with the constant functions $\v_h := \constvec{b}_i \in \KK_h[\u_h]$ for $i = 1, \ldots, m$.
	This yields that $b_h(\constvec{b}_i, \constvec{\lambda}) = 0$ for all $i = 1, \ldots, m$ and, hence, $\constvec{\lambda} = \constvec{0} \in \R^m$.
	Thus, we see that $b_h(\v_h, \constvec{\lambda}) = 0$ for all $\v_h \in \KK_h[\u_h]$ and, by plugging this into~\eqref{eq:saddle-point}, that $\w_h$ solves~\eqref{eq:discreteELbilinear}.
\end{proof}

\begin{remark}
	Two comments on the practical implementation of the discrete system~\eqref{eq:discreteELbilinear} with additional constraints are in order.

	\textbf{\upshape (i)} 
	First, recall from ~\cite{MR3995784} how the solution to equation~\eqref{eq:discreteELbilinear} can be computed using Householder matrices:
	Let $\constvec{u} \in \R^3$ be a vector with $|\constvec{u}| = 1$.
	Then, if $\constvec{u} = -\e_3$, define the diagonal matrix $\constvec{Q}[\constvec{u}] := \mathrm{diag}(-1,-1,1) \in \R^{3 \times 3}$, else define, with the identity matrix $\constvec{I}_{3 \times 3} \in \R^{3 \times 3}$,
	\begin{equation*}
		\constvec{Q}[\constvec{u}] := \constvec{I}_{3 \times 3} - 2 \constvec{v} \constvec{v}^\transposed \in \R^{3 \times 3},
		\quad \text{where} \quad
		\constvec{v} := \frac{\constvec{u} + \e_3}{|\constvec{u} + \e_3|} \in \R^3 \quad
		\text{with } |\constvec{v}| = 1.
	\end{equation*}
	Note that $\constvec{Q}[\constvec{u}]$ is orthogonal and $(\constvec{Q}[\constvec{u}]\e_i)_{i=1}^3$ is an orthonormal basis of $\R^3$ with $\constvec{Q}[\constvec{u}] \e_3 = -\constvec{u}$ and, consequently, $\constvec{Q}[\constvec{u}] \e_i \perp \constvec{u}$ for $i = 1, 2$.
	We then define the prolongation operator $\QQ[\constvec{u}] \colon \R^2 \to \R^3$ by
	\begin{equation*}
		\QQ[\constvec{u}] \widehat{\constvec{w}}
		:=
		\sum_{i=1}^2 \widehat{w}_i \constvec{Q}[\constvec{u}] \e_i
		\quad
		\text{for all } \widehat{\constvec{w}} \in \R^2
	\end{equation*}
	and, for $\u_h \in \MM_h$, the discrete prolongation operator
	\begin{equation*}
		\QQ_h[\u_h] \colon [\SS^1(\TT_h)]^2 \to \vec{\SS}^1(\TT_h) \colon
		\QQ_h[\u_h] \widehat{\w}_h
		:=
		\sum_{\z \in \NN_h} \QQ[\u_h(\z)] \widehat{\w}_h(\z) \varphi_{\z}.
	\end{equation*}
	These definitions yield that $\QQ_h[\u_h] \widehat{\w}_h \in \KK_h[\u_h]$ for all $\w_h \in [\SS^1(\TT_h)]^2$.
	Thus, the discrete Euler--Lagrange equations~\eqref{eq:discreteELbilinear} can be transformed to a system of equations in $[\SS^1(\TT_h)]^2$ via the bijective map $\QQ_h[\u_h] \colon [\SS^1(\TT_h)]^2 \to \KK_h[\u_h]$:
	find $\widehat{\w}_h \in [\SS^1(\TT_h)]^2$ such that
	\begin{equation*}
		a_h(\QQ_h[\u_h] \widehat{\w}_h, \QQ_h[\u_h] \widehat{\v}_h)
		=
		a_h(\u_h, \QQ_h[\u_h] \widehat{\v}_h)
		\quad \text{for all } \widehat{\v}_h \in [\SS^1(\TT_h)]^2,
	\end{equation*}
	and set $\w_h := \QQ_h[\u_h] \widehat{\w}_h$.
	This system of $2\#\NN_h$ equations can be solved efficiently by iterative solvers, e.g., the generalized minimal residual method (GMRES) from~\cite{ss1986}.

	\textbf{\upshape (ii)} 
	With $B$ from Proposition~\ref{prop:numerical-uniqueness}, in most practical situations there holds that $\mathrm{span} \, B = \R^3$ so that the tangential update $\w_h$ is unique without additional constraints.
	It is easy to check the dimension of $\mathrm{span} \, B$ to determine whether additional constraints are necessary.

	E.g., if $\gamma = 0$, take an arbitrary node $\z^0 \in \NN_h$ and compute $\constvec{p}_{\z} := \u_h(\z^0) \times \u_h(\z)$ for all $\z \in \NN_h$.
	If $\constvec{p}_{\z} = \constvec{0}$ for all $\z \in \NN_h$, then $B_\perp$ consists of any two orthonormal vectors orthogonal to $\u_h(\z^0)$ and two constraints need to be added to~\eqref{eq:discreteELbilinear}.
	Otherwise, let $\z^1 \in \NN_h$ be any node with $\constvec{p}_{\z^1} \neq \constvec{0}$.
	If $\constvec{p}_{\z^1} \cdot \u_h(\z) = 0$ for all $\z \in \NN_h$, then $B_\perp = \{\constvec{p}_{\z^1} / |\constvec{p}_{\z^1}|\}$ and one constraint needs to be added to~\eqref{eq:discreteELbilinear}.
	Otherwise, $B_\perp = \{\constvec{0}\}$ and no additional constraint is necessary.
	
	The cases $\gamma > 0$ and $\gamma < 0$ can be treated analogously.
\end{remark}

\subsection{Proof of Theorem~\ref{th:convergence-discrete}}\label{sec:discrete}

To prove convergence of Algorithm~\ref{alg:alouges-discrete}, we intend to mimic the argument for the proof of convergence of Algorithm~\ref{alg:alouges}. First, we ensure that both steps of Algorithm~\ref{alg:alouges-discrete} decrease the discrete energy $\JJ_h$. We start with step~(i).

\begin{lemma}\label{lemma:discrete-energy-decrease}
	Let $\u_h \in \vec{S}^1(\TT_h)$ and $\w_h \in \KK_h[\u_h]$ such that $a_h(\w_h, \v_h) = a_h(\u_h, \v_h)$ for all $\v_h \in \KK_h[\u_h]$.
	Then, it holds that
	\begin{equation}
	\label{eq:discrete-energy-decrease}
		\JJ_h(\u_h - \w_h)
		=
		\JJ_h(\u_h) - \JJ_h(\w_h) \le \JJ_h(\u_h).
	\end{equation}
\end{lemma}

\begin{proof}
	The proof is analogous to the continuous case; see Lemma~\ref{lemma:continuous-energy-decrease}.
\end{proof}

In contrast to the continuous case, $|\u_h - \w_h| \geq 1$ is not guaranteed to hold true except at vertices $\z \in \NN_h$.
Therefore, the analysis of the projection in step~\textrm{(iii)} requires a different approach from the continuous case.
We note that the angle condition is crucial for treating the exchange energy~\cite[Lemma~3.2]{bartels05}.
For the convenience of the reader, we include the main argument of~\cite{bartels05} in the proof of the next lemma.

\begin{lemma}\label{lemma:discrete-projection-decrease}
	Let $\kappa = 0$ and let $\TT_h$ satisfy the angle condition~\eqref{eq:angle-condition}. 
	Then, it holds that
	\begin{equation}
	\label{eq:discrete-projection-decrease}
		\JJ_h(\Pi_h(\v_h))
		\leq
		\JJ_h(\v_h)
		\qquad \text{for all }
		\v_h \in \MM_h^+.
	\end{equation}
\end{lemma}

\begin{proof}
	The proof is split into two steps.
	
	\textbf{Step 1:} We recall the argument of~\cite{bartels05} to show that
	\begin{equation}
	\label{eq:decrease-exchange}
		\int_\Omega | \nabla (\Pi_h \v_h) |^2 \d{x}
		\leq
		\int_\Omega | \nabla \v_h |^2 \d{x}.
	\end{equation} 
	Note that the angle condition~\eqref{eq:angle-condition} implies that the off-diagonal entries 
	\begin{equation*}
		A_{\z\z'}
		:=
		\int_{\Omega} \nabla \varphi_{\z} \cdot \nabla \varphi_{\z'} \d{x}
		=
		- (\cot \alpha_1 + \cot \alpha_2) \le 0
	\end{equation*}
	of the stiffness matrix $A$ are non-positive.
	For any function $\v_h \in \vec{\SS}^1(\TT_h)$, we have that
	\begin{equation*}
		\int_\Omega |\nabla \v_h|^2 \d{x}
		=
		\sum_{\z,\z' \in \NN_h} \v_h(\z) \cdot \v_h(\z') \int_\Omega \nabla \varphi_{\z} \cdot \nabla \varphi_{\z'} \d{x}
		=
		\sum_{\z,\z' \in \NN_h} A_{\z\z'} \v_h(\z) \cdot \v_h(\z').
	\end{equation*}
	From the definition of the stiffness matrix $A$, we infer that there holds
	\begin{equation*}
		\sum_{\z \in \NN_h} A_{\z\z'}
		=
		\int_\Omega \nabla 1 \cdot \nabla \varphi_{\z'} \d{x}
		=
		0
		\qquad
		\text{for all } \z' \in \NN_h.
	\end{equation*}
	This leads to
	\begin{equation*}
		\int_\Omega |\nabla \v_h|^2 \d{x}
		=
		\!\!\!\sum_{\z,\z' \in \NN_h}\!\! A_{\z\z'} \v_h(\z) \cdot \v_h(\z') - \!\!\!\sum_{\z,\z' \in \NN_h}\!\! A_{\z\z'}  \v_h(\z')^2
		=
		\!\!\!\sum_{\z,\z' \in \NN_h}\!\! A_{\z\z'} \big( \v_h(\z) - \v_h(\z') \big) \cdot \v_h(\z').
	\end{equation*}
	Therefore, symmetry of the stiffness matrix leads to
	\begin{align*}
		\int_\Omega |\nabla \v_h|^2 \d{x}
		&=
		\frac{1}{2} \!\sum_{\z,\z' \in \NN_h}\!\! A_{\z\z'} \big( \v_h(\z) - \v_h(\z') \big) \cdot \v_h(\z')
		+ \!\frac{1}{2} \!\sum_{\z,\z' \in \NN_h}\!\! A_{\z\z'} \big( \v_h(\z') - \v_h(\z) \big) \cdot \v_h(\z).\\
		&=
		- \frac{1}{2} \sum_{\z,\z' \in \NN_h} A_{\z\z'} \big| \v_h(\z) - \v_h(\z') \big|^2
		=
		- \frac{1}{2} \sum_{\substack{\z,\z' \in \NN_h \\ \z \neq \z'}} A_{\z\z'} \big| \v_h(\z) - \v_h(\z') \big|^2 
	\end{align*}
	Note that the mapping $\vec{x} \in \set{\tilde{\vec{x}} \in \R^3}{|\tilde{\vec{x}}| \geq 1} \mapsto \vec{x}/|\vec{x}|$ is Lipschitz continuous with constant $1$.
	With $A_{\z\z'} \leq 0$ for all $\z \neq \z'$ and $|\v_h(\z)| \geq 1$ for all $\z \in \NN_h$, we thus have 
	\begin{align*}
		\int_\Omega |\nabla (\Pi_h \v_h)|^2 \d{x}
		&=
		- \frac{1}{2} \sum_{\substack{\z,\z' \in \NN_h \\ \z \neq \z'}} A_{\z\z'} \left| \frac{\v_h(\z)}{|\v_h(\z)|} - \frac{\v_h(\z')}{|\v_h(\z')|} \right|^2\\
		&\leq
		- \frac{1}{2} \sum_{\substack{\z,\z' \in \NN_h \\ \z \neq \z'}} A_{\z\z'} \left| \v_h(\z) - \v_h(\z') \right|^2
		=
		\int_\Omega |\nabla \v_h|^2 \d{x}
	\end{align*}
	This shows~\eqref{eq:decrease-exchange}.

	\textbf{Step 2:} We show that
	\begin{equation}
	\label{eq:decrease-Spart}
		\int_\Omega \II_h \big( g_\gamma(\Pi_h \v_h, \Pi_h \v_h) \big) \d{x}
		\leq
		\int_\Omega \II_h \big( g_\gamma(\v_h, \v_h) \big) \d{x}.
	\end{equation}
	Since $|\v_h(\z)| \geq 1$ for all $\z \in \NN_h$ and since $g_\gamma(\cdot, \cdot)$ is bilinear, we have that
	\begin{align*}
		\int_\Omega \II_h \big( g_\gamma(\Pi_h \v_h, \Pi_h \v_h) \big) \d{x}
		&=
		\sum_{\z \in \NN_h} \frac{g_\gamma(\v_h(\z), \v_h(\z))}{|\v_{h}(\z)|^2} \int_\Omega \varphi_{\z} \d{x}\\
		&\leq
		\sum_{\z \in \NN_h} g_\gamma(\v_h(\z), \v_h(\z)) \int_\Omega \varphi_{\z} \d{x}
		=
		\int_\Omega \II_h \big( g_\gamma(\v_h, \v_h) \big) \d{x}.
	\end{align*}
	This shows~\eqref{eq:decrease-Spart}.
	Combining~\eqref{eq:decrease-exchange}--\eqref{eq:decrease-Spart}, we conclude the claim for $\kappa = 0$.
\end{proof}

\begin{remark}
    The proof of Lemma~\ref{lemma:discrete-projection-decrease} relies on two facts that are specific to the case $\kappa = 0$:
	The angle condition on the triangulation implies the non-positivity of the off-diagonal entries of the stiffness matrix;
	combined with the $1$-Lipschitz property of the nodal normalization $x \mapsto
	x / |x|$, this yields the exchange-energy estimate
	\begin{equation*}
		\int_{\Omega} | \nabla (\Pi_h \v_h) |^2
		\leq
		\int_{\Omega} | \nabla \v_h |^2 .
	\end{equation*}
	By contrast, when $\kappa \neq 0$, the helical (DMI) contribution produces
	mixed and lower-order terms that at the discrete level are not controlled by the angle condition alone;
	consequently, the total discrete energy may increase after projection, as we show in Section~\ref{sec:helical-increase}.
	Determining mesh or projection conditions that restore discrete monotonicity for $\kappa \neq 0$ remains an open question.

	Nevertheless, the observed increases in the energy level are very small and sporadic in our tests and Algorithm~\ref{alg:alouges-discrete} empirically converges also for $\kappa \neq 0$;
	see Section~\ref{sec:reduced-model} for the numerical experiments.
\end{remark}

For the proof of Theorem~\ref{th:convergence-discrete}, the next lemma collects some basic properties of the nodal projection.

\begin{lemma}
\label{lemma:nodal-projection}
	Recall the nodal projection $\II_h \colon C^0(\overline{\Omega}) \to \SS^1(\TT_h)$ from the definition of $\JJ_h$ in~\eqref{eq:energy-discrete}.
	Then, there hold the following properties {\rm (i)--(iii)}.
	\begin{enumerate}[label={\rm (\roman*)}]
		\item For all $v \in C^2(\overline{\Omega})$ and all $1 \leq p < \infty$, there holds the approximation property
		\begin{equation}
		\label{eq:nodal-approximation}
			\norm{(1-\II_h) v}_{L^p(\Omega)} + h \norm{\nabla (1-\II_h) v}_{L^p(\Omega)}
			\le C_{\rm apx}
			h^2 \norm{D^2 v}_{L^p(\Omega)}.
		\end{equation}
		
		\item For all $v_h \in \SS^q(\Omega)$ with $q \in \N$ and all $1 \leq p < \infty$, there holds the inverse inequality
		\begin{equation}
		\label{eq:inverse-inequality}
			h \norm{D^2 v_h}_{L^p(\Omega)}
			\le C_{\rm inv}
			\norm{\nabla v_h}_{L^p(\Omega)}.
		\end{equation}
		
		\item For all $v_h \in \SS^1(\Omega)$, there holds
		\begin{equation}
		\label{eq:nodal-l2}
			\norm{v_h}_{L^2(\Omega)}^2
			\leq
			\int_{\Omega} \II_h ( |v_h|^2 ) \d{x}.
		\end{equation}
	\end{enumerate}
    The constant $C_{\rm apx} > 0$ depends only on shape regularity of $\TT_h$ and the Lebesgue index $p$, while $C_{\rm inv} > 0$ depends only on shape regularity of $\TT_h$, the Lebesgue index $p$, and the polynomial degree $q$.
\end{lemma}

\begin{proof}
	For the proof of properties {\rm (i)--(ii)} we refer to standard literature, e.g.,~\cite{MR2373954}.
	For the proof of~{\rm (iii)}, we note that on each element $T \in \TT_h$, there holds the representation $v_h(x) = \sum_{\z \in \NN_h(T)} \varphi_{\z}(x) v_h(\z)$, where $\sum_{\z \in \NN_h(T)} \varphi_{\z}(x) = 1$ and $\varphi_{\z}(x) \geq 0$ for all $x \in T$.
	Thus, by Jensen's inequality, we have that
	\begin{equation*}
		\norm{v_h}_{L^2(T)}^2
		=
		\int_{T} \big| \sum_{\z \in \NN(T)} \varphi_{\z} v_h(\z) \big|^2 \d{x}
		\leq
		\int_{T} \sum_{\z \in \NN(T)} \varphi_{\z} |v_h(\z)|^2 \d{x}
		=
		\int_{T} \II_h ( |v_h|^2 ) \d{x}.
	\end{equation*}
	Summing over all elements $T \in \TT_h$, we conclude the proof.
\end{proof}

Finally, we can prove the second part of Theorem~\ref{th:convergence-discrete}, i.e., convergence of the output of Algorithm~\ref{alg:alouges-discrete} to a weak solution of the Euler--Lagrange equations~\eqref{eq:el-equations}.

\begin{proof}[{\bfseries Proof of Theorem~\ref{th:convergence-discrete}}]
	We split the proof into seven steps.

\textbf{Step 1 (Algorithm~\ref{alg:alouges-discrete} is well-defined):}
	The proof of existence of at least one solution $\w_h^{n}$ to the discrete variational formulation~\eqref{eq:discreteELbilinear} in step~\textrm{(i)} of Algorithm~\ref{alg:alouges-discrete} follows analogously to the continuous case;
    see step~1 in the proof of Theorem~\ref{th:convergence}.
	For any $n \geq 0$, the fact that $\w_{h}^{n} \in \KK_h[\u_h^{n}]$ yields that 
    \begin{align*}
        |\u_h^n(\z) - \w_h^n(\z)|^2 = |\u_h^n(\z)|^2 + |\w_h^n(\z)|^2 
        \ge 1 \quad \text{for all vertices } \z \in \NN_h.
    \end{align*}%
    Therefore, for any $\kappa \in \R$, there holds $\u_h^n - \w_h^n \in \MM_h^+$ and Step~\textrm{(iii)} in Algorithm~\ref{alg:alouges-discrete} is always well-defined.

\textbf{Step 2 (Energy decrease along computed sequence):}
From Lemma~\ref{lemma:discrete-energy-decrease} and Lemma~\ref{lemma:discrete-projection-decrease} (exploiting $\kappa = 0$), we obtain that
	\begin{equation*}
		\JJ_h \big( \u_h^{n+1} \big)
		=
		\JJ_h \big( \Pi_h(\u_h^{n} - \w_h^{n}) \big)
		\eqreff{eq:discrete-energy-decrease}{\leq}
		\JJ_h \big( \u_h^{n} - \w_h^n\big)
		\leq
		\JJ_h \big( \u_h^{n} \big)
        \quad \text{for all } n \in \N_0
	\end{equation*}
	and thus conclude $\JJ_h(u_h) \leq \JJ_h(u_h^0)$ by induction.

\textbf{Step 3 (Algorithm~\ref{alg:alouges-discrete} terminates):}
Arguing as for~\eqref{eq:w-convergence} in step~3 of the proof of Theorem~\ref{th:convergence}, we conclude that $\JJ_h(\w_h^{n}) \to 0$ as $n \to \infty$. In particular, we have that $\JJ_h(\w_h^{n}) \leq \mathrm{tol}_h$ for some $n \in \N_0$ so that Algorithm~\ref{alg:alouges-discrete} terminates. this concludes the proof of Theorem~\ref{th:convergence-discrete}(i).

\textbf{Step 4 (Existence of weakly convergent subsequence):} Having proved Theorem~\ref{th:convergence-discrete}(i),
Lemma~\ref{lemma:discrete-energy-decrease} yields that
	\begin{equation*}
		\norm{\u_h}_{\vec{H}^1(\Omega)}^2
		\lesssim
		\JJ_h(\u_h) + \norm{\u_h}_{\vec{L}^2(\Omega)}^2
		\leq
		\JJ_h(\u_h^{0}) + \norm{\u_h}_{\vec{L}^2(\Omega)}^2
		\eqreff{eq:nodal-l2}{\leq}
		C_0 + |\Omega|.
	\end{equation*}
	Therefore, there exists a subsequence of $(\u_h)_h$ (which is not relabeled) and $\u \in \vec{H}^1(\Omega)$ such that $\u_h \weak \u$ weakly in $\vec{H}^1(\Omega)$.

\textbf{Step 5 (Weak limit $\u$ satisfies modulus constraint):}
	Since $|\u_h(\z)| = 1$ for all $\z \in \NN_h$, we have that $\II_h(|\u_h|^2) = 1 \in \SS^1(\TT_h)$.
	Moreover, $|\u_h|^2 = \u_h \cdot \u_h \in [\SS^2(\TT_h)]^3$.
	With Lemma~\ref{lemma:nodal-projection}, we get on each element $T \in \TT_h$ that
	\begin{align*}
		&\norm{|\u_h|^2 - 1}_{\vec{L}^2(T)}
		=
		\norm{(1-\II_h)|\u_h|^2}_{\vec{L}^2(T)}
		\eqreff{eq:nodal-approximation}{\lesssim}
		h^2 \norm{D^2 |\u_h|^2}_{\vec{L}^2(T)}
		\eqreff{eq:inverse-inequality}{\lesssim}
		h \norm{\nabla |\u_h|^2}_{\vec{L}^2(T)}\\
		& \quad =
		2 h \norm{\u_h \cdot \nabla \u_h}_{\vec{L}^2(T)}
		\leq
		2 h \norm{\u_h}_{\vec{L}^\infty(\Omega)} \norm{\nabla \u_h}_{\vec{L}^2(T)} \\
		& \quad =
		2 h \norm{\nabla \u_h}_{\vec{L}^2(T)}
		\lesssim
		h \JJ_h(\u_h)^{1/2}
		\lesssim
		h \to 0.
	\end{align*}
	Since $\u_h \rightharpoonup \u$ in $\vec{H}^1(\Omega)$ and $\u_h \to \u$ in $\vec{L}^2(\Omega)$, there holds (up to a subsequence which is not relabeled) that $\u_h \to \u$ almost everywhere in $\Omega$.
	This implies $|\u| = 1$ a.e.\ in $\Omega$ and hence $\u \in \MM$.
	
	\textbf{Step 6 ($\u$ satisfies the Euler--Lagrange equations~(\ref{eq:el-equations}) of $\JJ$):}
	Let $n \in \N_0$ be the index such that $\u_h = \u_h^{n}$ is returned by Algorithm~\ref{alg:alouges-discrete}.
	Defining $\w_h := \w_h^{n}$, for all $\v_h \in \KK_h[\u_h]$, it follows that
	\begin{equation*}
		\int_\Omega \nabla (\u_h - \w_h) : \nabla \v_h
			+ \II_h \big( g_\gamma(\u_h - \w_h, \v_h) \big) \d{x}
		\eqreff{eq:energy-discrete}{=}
		a_h(\u_h - \w_h, \v_h)
		\eqreff{eq:discreteELbilinear}{=}
		0.
	\end{equation*}
	For $\vec{\varphi} \in \vec{C}^\infty(\Omega)^3$, define $\overline{\v} := \vec{\varphi} \times \u_h \in \vec{H}^1(\Omega) \cap \vec{C}(\Omega)$ and $\overline{\v}_h := \II_h(\overline{\v}) \in \vec{\SS}^1(\TT_h)$.
	Then, the last equation gives
	\begin{align}
	\nonumber
		\int_\Omega \Big[ \nabla \u_h &: \nabla \overline{\v}
			+ \II_h \big( g_\gamma(\u_h, \overline{\v}) \big) \Big] \d{x}\\
	\label{eq:convergence-discrete-start}
		&=
		\int_\Omega \Big[ \nabla (\u_h - \w_h) : \nabla (\overline{\v} - \overline{\v}_h)
			+ \II_h \big( g_\gamma(\u_h - \w_h, \overline{\v} - \overline{\v}_h) \big) \Big] \d{x}\\
	\nonumber
		& \qquad + \int_\Omega \Big[ \nabla \w_h : \nabla \overline{\v}
			+ \II_h \big( g_\gamma(\w_h, \overline{\v}) \big) \Big] \d{x}.
	\end{align}
	It remains to show that, as $h \to 0$, the left-hand side and right-hand side of~\eqref{eq:convergence-discrete-start} converge to the left-hand side and right-hand side of~\eqref{eq:el-equations}, respectively.
	Then, Lemma~\ref{lemma:el-equations} yields that $\u$ is a weak solution of the Euler--Lagrange equations~\eqref{eq:el-strong}.
	
	\textbf{Step 6.1 (Left-hand side of~(\ref{eq:convergence-discrete-start}) converges to left-hand side of~(\ref{eq:el-equations}) as $\boldsymbol{h \to 0}$):}	
	We notice that
	\begin{equation*}
		\int_\Omega \Big[ \nabla \u_h : \nabla \overline{\v} + \II_h \big( g_\gamma(\u_h, \overline{\v}) \big) \Big] \d{x}
		=
		\int_\Omega \Big[ (\u_h \times \nabla \u_h) : \nabla \vec{\varphi} + \II_h \big( g_\gamma(\u_h, \vec{\varphi} \times \u_h) \big) \Big] \d{x}.
	\end{equation*}
	Recall that $\nabla \u : \nabla (\vec{\varphi} \times \u) = \u \times \nabla \u : \nabla \vec{\varphi}$.
	Since $\u_h \weak \u$ in $\vec{H}^1(\Omega)$ implies $\u_h \to \u$ in $\vec{L}^2(\Omega)$ and $\nabla \u_h \weak \nabla \u$ in $\vec{L}^2(\Omega)$, there holds for the first term that
	\begin{equation*}
		\int_\Omega \nabla \u_h : \nabla \overline{\v} \d{x}
		=
		\int_\Omega \nabla \u_h : \nabla \vec{\varphi \times \u_h} \d{x}
		=
		\int_\Omega (\u_h \times \nabla \u_h) : \nabla \vec{\varphi} \d{x}
		\to
		\int_\Omega (\u \times \nabla \u) : \nabla \vec{\varphi} \d{x}.
	\end{equation*}
	For the mass lumping term, we first notice that
	\begin{equation}
	\label{eq:convergence-discrete-mass}
		\int_{\Omega} \II_h(v^h) \d{x}
		=
		\int_\Omega v \d{x} + \int_{\Omega} v^h - v \d{x} + \int_{\Omega} (\II_h - 1) v^h \d{x}.	
	\end{equation}
	Since $\norm{\u}_{\vec{L}^\infty(\Omega)} = \norm{\u_h}_{\vec{L}^\infty(\Omega)} = 1$ , the choices $v^h := g_\gamma(\u_h, \vec{\varphi} \times \u_h)$ and $v := g_\gamma(\u, \vec{\varphi} \times \u)$ yield that
	\begin{equation*}
		\int_{\Omega} |v^h - v| \d{x}
		\lesssim
		(\norm{\u_h}_{\vec{L}^\infty(\Omega)} + \norm{\u}_{\vec{L}^\infty(\Omega)}) \norm{\u_h-\u}_{\vec{L}^2(\Omega)} \norm{\vec{\varphi}}_{\vec{L}^2(\Omega)}
		\lesssim
		\norm{\u_h-\u}_{\vec{L}^2(\Omega)}
		\to 0.
	\end{equation*}
	The third term can be treated by the local approximation property of the nodal interpolation on each element, which yields that
	\begin{equation*}
		\int_T \big| (\II_h - 1) v^h \big| \d{x}
		\leq
		|T|^{1/2} \norm{(\II_h - 1) v^h}_{L^2(T)}
		\lesssim
		|T|^{1/2} h^2 \norm{D^2 v^h}_{L^2(T)}.
	\end{equation*}
	Looking at $v^h = g_\gamma(\u_h, \vec{\varphi} \times \u_h)$ in detail, we see that, regardless of the sign of $\gamma$, this expression is a sum of products of the form $u_{h,i}u_{h,j}\varphi_k$ for some $i,j,k \in \{1,2,3\}$.
	For the Hessian of such a product, we obtain with $\u_h \in \SS^1(\TT_h)$, $\norm{\u_h}_{\vec{L}^\infty(\Omega)} = 1$, and $\norm{\nabla \u_h}_{\vec{L}^2(\Omega)} \lesssim 1$ that
	\begin{align}
	\nonumber
		&\norm{D^2 (u_{h,i}u_{h,j}\varphi_k)}_{\vec{L}^2(T)}\\
		& \qquad =
	\label{eq:convergence-discrete-explicit}
		\norm{u_{h,i} u_{h,j} D^2 \varphi_k +
			\varphi_k \nabla u_{h,i} \nabla u_{h,j}^\transposed +
			u_{h,j} \nabla u_{h,i} \nabla \varphi_k^\transposed +
			u_{h,i} \nabla u_{h,j} \nabla \varphi_k^\transposed}_{\vec{L}^2(T)}\\
	\nonumber
		& \qquad \lesssim
		\big(
			\norm{\u_h}_{\vec{L}^2(\Omega)}^2
			+ \norm{\nabla \u_h}_{\vec{L}^2(\Omega)}^2
			+ \norm{\u_h}_{\vec{L}^2(\Omega)} \norm{\nabla \u_h}_{\vec{L}^2(\Omega)}
		\big) \norm{\vec{\varphi}}_{\vec{W}^{2,\infty}(T)}
		\lesssim
		1.
	\end{align}
	Combining the estimates~\eqref{eq:convergence-discrete-mass}--\eqref{eq:convergence-discrete-explicit} and summing over all elements, we conclude that
	\begin{equation}
	\label{eq:convergence-discrete-mass-lumping}
		\int_\Omega \II_h \big( g_\gamma(\u_h, \vec{\varphi} \times \u_h) \big) \d{x}
		\quad\stackrel{h \to 0}{\longrightarrow}\quad
		\int_\Omega g_\gamma(\u, \vec{\varphi} \times \u) \d{x}.
	\end{equation}
	
	\textbf{Step 6.2 (Right-hand side of~(\ref{eq:convergence-discrete-start}) vanishes as $\boldsymbol{h \to 0}$):}	
	We first observe that there holds $\overline{\v}(\z) - \overline{\v}_h(\z) = \vec{0}$ for all $z \in \NN_h$ and, thus, the term $\II_h((u_{h,3} - w_{h,3}) (\overline{v}_{3} - \overline{v}_{h,3}))$ vanishes.
	Using the elementwise interpolation property of the nodal interpolation, we further see that
	\begin{equation*}
		\norm{\nabla(\overline{\v} - \overline{\v}_h)}_{\vec{L}^2(T)}
		=
		\norm{\nabla(I-\II_h)(\vec{\varphi} \times \u_h)}_{\vec{L}^2(T)}
		\lesssim
		h \norm{D^2 (\vec{\varphi} \times \u_h)}_{\vec{L}^2(T)}.
	\end{equation*}
	An explicit computation in the spirit of~\eqref{eq:convergence-discrete-explicit} shows that $\norm{D^2 (\vec{\varphi} \times \u_h)}_{\vec{L}^2(T)} \lesssim \norm{\vec{\varphi}}_{\vec{W}^{2,\infty}(T)}$.
	Thus, summing over all elements, we get that
	\begin{equation*}
		\norm{\nabla(\overline{\v} - \overline{\v}_h)}_{\vec{L}^2(\Omega)}
		\lesssim
		h \norm{\vec{\varphi}}_{\vec{W}^{2,\infty}(\Omega)} \to 0
		\quad
		\text{as } h \to 0.
	\end{equation*}
	Recall that
	\begin{equation*}
	\norm{\nabla(\u_h - \w_h)}_{\vec{L}^2(\Omega)}
	\leq
	\norm{\nabla\u_h}_{\vec{L}^2(\Omega)} + \norm{\w_h}_{\vec{L}^2(\Omega)}
	\lesssim
	\JJ_h(\u_h)^{1/2} + \JJ_h(\w_h)^{1/2}
	\leq
	C_0^{1/2} + \mathrm{tol}_h^{1/2}
	\end{equation*}
	is uniformly bounded.
	Therefore, we have that 
    $$
    \int_\Omega \big[ \nabla (\u_h - \w_h) : \nabla (\overline{\v} - \overline{\v}_h) + \II_h(g_\gamma(\u_h - \w_h, \overline{\v} - \overline{\v}_h)) \big] \d{x} \to 0
    \quad \text{as $h \to 0$.}
    $$
	With $\norm{\nabla \w_h}_{\vec{L}^2(\Omega)}^2 \leq \mathrm{tol}_h \to 0$ and the boundedness of $\overline{\v}$ in $\vec{H}^1(\Omega) \cap \vec{C}(\Omega)$, the Hölder inequality proves
	\begin{equation*}
		\int_\Omega \big[ \nabla \w_h : \nabla \overline{\v}
		+ \II_h\big( g_\gamma(\w_h, \overline{\v}) \big) \big] \d{x}
		\lesssim
		\JJ_h(\w_h)
		\lesssim
		\mathrm{tol}_h
		\to 0.
	\end{equation*}
	
	\textbf{Step~7 (Proof of energy estimate~(\ref{eq:convergence-discrete})):}
Analogously to~\eqref{eq:convergence-discrete-mass-lumping}, it follows that
	\begin{equation*}
		\int_\Omega \II_h \big( g_\gamma(\u_h, \u_h) \big) \d{x}
		\quad\stackrel{h \to 0}{\longrightarrow}\quad
		\int_\Omega g_\gamma(\u, \u) \d{x}.
	\end{equation*}
	Therefore, it follows that
	\begin{align*}
		\liminf_{h \to 0} \JJ_h(\u_h)
		&=
		\liminf_{h \to 0} \int_\Omega \nabla \u_h : \nabla \u_h \d{x}
		+ \lim_{h \to 0} \int_\Omega \II_h \big( g_\gamma(\u_h, \u_h) \big) \d{x}\\
		&\geq
		\int_\Omega \big[ \nabla \u : \nabla \u + g_\gamma(\u, \u) \big] \d{x}
		=
		\JJ(\u).
	\end{align*}
	This concludes the proof.
\end{proof}

\begin{remark}
We note that $\kappa = 0$ is only exploited in the crucial step~2 of the preceding proof of Theorem~\ref{th:convergence-discrete}, where Lemma~\ref{lemma:discrete-projection-decrease} is used. All other steps are valid for arbitrary $\kappa \in \R$.
\end{remark}%

\subsection{Necessity of nodal projection}\label{sec:anisotropy-increase}

Note that using mass lumping for the discrete energy is necessary in order to prove Lemma~\ref{lemma:discrete-projection-decrease}.
Without the nodal projection, energy decrease for the $g_\gamma$-term (see~\eqref{eq:decrease-Spart}) does not hold true in general.
In particular, for $\gamma \geq 0$, consider the unit triangle $T := \conv \{ (0,0), (1,0), (0,1) \}$, where $\conv$ denotes the convex hull.
An explicit computation with $\varepsilon > 0$ and $\delta := \sqrt{2-\varepsilon^2/2}$ shows that the discrete function 
\begin{equation*}
	\v_h := \colvec{\delta}{\delta}{-\varepsilon} (1-x-y) + \colvec{0}{0}{1} x + \colvec{0}{0}{1} y ~\in \SS^1(T)
\end{equation*}
satisfies $|\v_h(0,0)| = 2$ and $\norm{\Pi_h \v_h \cdot \e_3}_{L^2(T)}^2 > \norm{\v_h \cdot \e_3}_{L^2(T)}^2$ for $\varepsilon < 4/3$, since
\begin{align*}
	\norm{\v_h \cdot \e_3}_{L^2(T)}^2
	&=
	\norm{(1+\varepsilon)x + (1+\varepsilon)y - \varepsilon}_{L^2(T)}^2
	=
	\frac{1}{48}\big( 12 - 8\varepsilon + 4\varepsilon^2 \big),\\
	\norm{\Pi_h \v_h \cdot \e_3}_{L^2(T)}^2
	&=
	\norm{(1+\varepsilon/2)x + (1+\varepsilon/2)y - \varepsilon/2}_{L^2(T)}^2
	=
	\frac{1}{48}\big( 12 - 4\varepsilon + \varepsilon^2 \big).
\end{align*}
For $\gamma < 0$, a similar counterexample can be constructed.

\subsection{Increase in helical energy for $\vec{\kappa} \neq \vec{0}$}\label{sec:helical-increase}

Since, analytically, it is not straightforward to either prove or to disprove energy decrease in every step of Algorithm~\ref{alg:alouges-discrete} in the case $\kappa \neq 0$ (see~\eqref{eq:discrete-energy-decrease}), we seek to decide on that matter numerically.
To this end, we look at the mesh $\TT_h = \{T\}$ consisting only of the unit triangle $T = \conv \{(0,0), (1,0), (0,1) \}$.
We stress that $\TT_h$ satisfies the angle condition~\eqref{eq:angle-condition}.

On this mesh, random vectors $\constvec{a}, \constvec{b}, \constvec{c} \in \R^3$ with $|\constvec{a}|, |\constvec{b}|, |\constvec{c}| > 1$ are generated to define
\begin{equation*}
	\v_h(x,y)
	:=
	(1-x-y)\,\constvec{a} + x\,\constvec{b} + y\,\constvec{c}
	\in \SS^1(\TT_h).
\end{equation*}
Subsequently, the energy difference $\JJ_h(\v_h) - \JJ_h(\Pi_h\v_h)$ is computed with Mathematica~\cite{mathematica}.
We found that, e.g., the values
\begin{equation*}
	\constvec{a}
	:=
	\left( \begin{array}{c}
		0.44353334 \\ 0.86741656 \\ 0.22558999
	\end{array} \right),
	\quad
	\constvec{b}
	:=
	\left( \begin{array}{c}
		0.46138525 \\ 0.63580881 \\ 0.61893662
	\end{array} \right),
	\quad
	\constvec{c}
	:=
	\left( \begin{array}{c}
		0.5304891 \\ 0.66534908 \\ -0.52526736
	\end{array} \right),
\end{equation*}
with $|\constvec{a}|, |\constvec{b}|, |\constvec{c}| > 1 + \num{5e-6}$ lead to an energy increase $\JJ_h(\Pi_h\v_h) \approx \JJ_h(\v_h) + \num{9.27e-6}$.

We stress that similar results can be found for different shapes and scalings of the triangle $T$, although these examples are scarce and the energy increase is small.
This suggests that a generalization of the energy decrease~\eqref{eq:discrete-energy-decrease} to $\kappa \neq 0$ cannot be true in general. However, since Algorithm~\ref{alg:alouges-discrete} appears to converge also for $\kappa \neq 0$, we expect that this is a shortcoming of our analysis which potentially can be overcome.

\subsection{Numerical experiment on reduced 3D model}\label{sec:reduced-model}

To test our algorithm in the case of $\kappa \neq 0$ (which is thus not covered by Theorem~\ref{th:convergence-discrete}), we reproduce an experiment from~\cite[Section~4.2]{hpp+19}.
This work considers the three dimensional micromagnetic model
\begin{equation}
\label{eq:3d-model}
	\EE_{\mathrm{3D}}(\u)
	=
	\int_\Omega A \Big[ |\nabla \u|^2 + D \, \vec{\pi}[\u] \cdot \u + K |\u \times \e_3|^2 + \frac{\mu_0 M_s^2}{2} |\vec{h}_\mathrm{d}[\u]|^2 \Big] \d{x}
\end{equation}
with the interfacial Dzyaloshinskii--Moriya interaction (DMI)
\begin{equation}
\label{eq:dmi}
	\vec{\pi}[\u]
	=
	\e_1 \times \partial_2 \u - \e_2 \times \partial_1 \u
	=
	 \left(\begin{array}{c}
		-\partial_1 u_3\\
		-\partial_2 u_3\\
		\partial_1 u_1 + \partial_2 u_2
	\end{array}\right).
\end{equation}
While~\cite{hpp+19} employs the so-called Landau--Lifshitz--Gilbert equation, which models the time-dependent evolution of the magnetization (as a generalized gradient flow for $\EE_{\mathrm{3D}}$), we compute the steady state of the system by energy minimization.
Furthermore, we consider the anti-symmetric exchange term $\curl(\u) \cdot \u$ instead of interfacial DMI $\vec{\pi}[\u] \cdot \u$; see Remark~\ref{rem:domain-walls} below.
The material parameters are those of cobalt, i.e., $A = \SI{1.5e-11}{\joule\per\meter}$, $K = \SI{8e5}{\joule\per\meter\tothe{3}}$, and $M_s = \SI{5.8e5}{\ampere\per\meter}$.
The strength of the anti-symmetric exchange term is varied in the range $D = 0, 1, \ldots, \SI{8e-3}{\joule\per\meter\tothe{2}}$.
The computational domain $\Omega$ is a disk of diameter $\SI{80}{\nano\meter}$ within the plane spanned by $\e_1$ and $\e_2$.

To obtain a two-dimensional model that fits in the framework of the present paper, we first non-dimensionalize the model in the sense of, e.g.,~\cite{diss-michele}:
With the exchange length $\lex = \sqrt{2A/(\mu_0 M_s^2)}$ and the vacuum permeability $\mu_0 = 4\pi \cdot \SI{E-7}{\henry\per\meter}$, we rescale $x \to x/\lex$.
Together with $\kappa = D/(\mu_0 M_s^2 \lex)$, passing to the thin-film limit as in~\cite{ddfpr2022} yields the two-dimensional energy
\begin{equation}
\label{eq:2d-model}
	\EE_{\mathrm{2D}}(\u)
	=
	\int_\Omega \Bigl[ \frac{1}{2} |\nabla \u|^2 + \kappa \curl\u \cdot \u + \frac{K}{\mu_0 M_s^2} |\u \times \e_3|^2 + \frac{1}{2} |\u \cdot \e_3|^2 \Bigr] \d{x}.
\end{equation}
Using~\eqref{eq:cross-dot-identity} to summarize the anisotropy terms and rewriting the first two terms as helical derivative, this energy reads as~\eqref{eq:energy} with $\kappa = D/(\mu_0 M_s^2 \lex)$ and $\gamma = 1 - 2K/(\mu_0 M_s^2) - \kappa^2$.

We perform two different experiments that differ only by the initial values
\begin{align}
\label{eq:ic1}
	\u^0 &= \e_3,\\
\label{eq:ic2}
	\u^0 &= \begin{cases}
	- \e_3 & r < \SI{15}{\nano\meter} - \varepsilon,\\
	\e_3 & r > \SI{15}{\nano\meter} + \varepsilon,
	\end{cases}
\end{align}
with $\varepsilon > 0$ small and a smooth transition layer such that $\u^0 \in \MM$ in both cases, which represent a constant and a skyrmion-like initial magnetization, respectively.
The results for~\eqref{eq:ic1} are shown in Figure~\ref{fig:constant-ic} and those for~\eqref{eq:ic2} in Figure~\ref{fig:jump-ic}.
For the constant initial magnetization~\eqref{eq:ic1}, the magnetization stays essentially uniform up to $D=6$; for $D=7,8$, horseshoe-like magnetic domains are visible.
For the skyrmion-like initial magnetization~\eqref{eq:ic2}, the magnetization relaxes to a uniform magnetization for $D=0,1,2$, but forms a skyrmion for $D=3, \ldots, 6$ and a so-called target skyrmion for $D=8$.
This is in total agreement with~\cite{hpp+19}.
However, for the skyrmion-like initial magnetization and $D=7$ we get a horseshoe-like state, whereas~\cite{hpp+19} gets a target skyrmion.
This may be explained by the heavy dependence of the bifurcation point between simple skyrmion and target skyrmion on the relaxation method as suggested by numerical evidence in~\cite{dfppr2022}, as well as the effects explained in Remark~\ref{rem:domain-walls} below.
Since our algorithm is designed around physical energy minimization without the need for user-supplied parameters, we cannot control the relaxation process in our algorithm.

\begin{remark}
\label{rem:domain-walls}
	We note that considering the two-dimensional curl
	\begin{equation*}
	\curl(\u)
	=
	\e_1 \times \partial_1 \u + \e_2 \times \partial_2 \u
	=
	\left(\begin{array}{c}
		\partial_2 u_3\\
		-\partial_1 u_3\\
		\partial_1 u_2 - \partial_2 u_1
	\end{array}\right)
	\end{equation*}
	instead of the DMI term $\vec{\pi}[\u]$ of the three dimensional model does not have any effects on the topological structure of the magnetization.
	However, it changes the type of the domain walls:
	In~\cite{hpp+19}, only Neel-type domain walls occur, whereas our model produces only Bloch-type walls.
\end{remark}

\begin{remark}
	Numerically, the initial configuration~\eqref{eq:ic2} is meta-stable for $D=0$ in the thin-film case.
	In the thin-film 2D model of~\cite{ddfpr2022}, the stray-field contribution is approximated by $|\u \cdot \e_3|^2$, which causes the stray field of $\u^0_h$ to be parallel to $\e_3$.
	Together with $\Delta \u^0_h$ and anisotropy also being parallel to $\e_3$, and $\e_3 \perp \KK_h[\u^0_h]$, the initial tangential update $\w_h$ vanishes.
	In 3D, however, the stray field would not be pointing purely in the $\e_3$ direction because of boundary effects, thus exerting a small force perpendicular to $\e_3$ and perturbing that initial state.
	For this reason, we introduced a small divergent vector field for the disc part $r < \SI{15}{\nano\meter}$ to artificially perturb the initial state:
	$\u^0(x,y) = \Pi_h \big(\varepsilon \cdot (x,y,0)^\transposed - \e_3 \big)$ for some small $\varepsilon > 0$.
	For $D=0$, this proved sufficient to prevent the system from being trapped in the initial meta-stable state.
	For $D \neq 0$, this modification is not necessary.
\end{remark}

\begin{figure}
	\begin{tikzpicture}
		\foreach \D in {0,1,...,8}{
			\node (A\D) at ({1.6*\D},0) {\includegraphics[height=4em]{plots/constant_D\D}};
			\node[below=of A\D.center] (B\D) {$D = \D$};
		}
		\node[right= of A8.center] (CB) {{\includegraphics[height=4em]{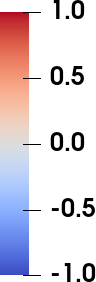}}};
	\end{tikzpicture}
	\caption{Simulation results for the model from Section~\ref{sec:reduced-model} with different values of $D$ and initial condition~\eqref{eq:ic1}.
		The color represents the contribution of the magnetization in the $\e_3$-direction.}
	\label{fig:constant-ic}
\end{figure}

\begin{figure}
	\begin{tikzpicture}
		\foreach \D in {0,1,...,8}{
			\node (A\D) at ({1.6*\D},0) {\includegraphics[height=4em]{plots/jump_D\D}};
			\node[below=of A\D.center] (B\D) {$D = \D$};
		}
		\node[right= of A8.center] (CB) {{\includegraphics[height=4em]{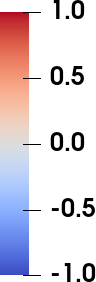}}};
	\end{tikzpicture}
	\caption{Simulation results for the model from Section~\ref{sec:reduced-model} with different values of $D$ and initial condition~\eqref{eq:ic2}.
	The color represents the contribution of the magnetization in the $\e_3$-direction.}
	\label{fig:jump-ic}
\end{figure}


\printbibliography

\end{document}